\documentclass[11pt, a4paper, oneside]{amsart}
\usepackage[T1]{fontenc}
\usepackage[utf8]{inputenc}
\usepackage[english]{babel}
\usepackage{lmodern}
\usepackage{exscale}
\usepackage[babel]{microtype}
\usepackage{amsmath, amssymb, mathtools, mathrsfs, amsthm}
\usepackage{thmtools}
\usepackage{hyperref}
\usepackage{comment}
\usepackage{fullpage}
\usepackage{enumitem}
\usepackage{todonotes}
\usepackage[capitalise,noabbrev]{cleveref}

\usepackage{graphicx}
\graphicspath{{./images/}}

\usepackage{tikz}
\usetikzlibrary{cd,shapes.geometric,positioning,patterns,calc}
\declaretheorem[numberwithin=section]{theorem}
\declaretheorem[sibling=theorem]{proposition}
\declaretheorem[sibling=theorem]{lemma}
\declaretheorem[sibling=theorem]{corollary}
\declaretheorem[sibling=theorem]{definition}
\declaretheorem[style=remark, sibling=theorem]{remark}

\newcommand{\tens}[1]{
\mathbin{\mathop{\otimes}\displaylimits_{#1}}
}
\newcommand{\field}{\mathbb{Q}}
\newcommand{\FB}{\textnormal{FB}}
\newcommand{\FI}{\textnormal{FI}}
\newcommand{\FIs}{\FI^{\#}}

\newcommand{\Ind}{\operatorname{Ind}}
\newcommand{\Res}{\operatorname{Res}}
\newcommand{\Hom}{\operatorname{Hom}}

\newcommand{\op}{\textnormal{op}}
\newcommand{\im}{\operatorname{im}}
\newcommand{\crosseffect}{\operatorname{\mathbf{cr}}}
\newcommand{\emb}[1]{\operatorname{Emb}({#1}, \mathbb{R}^d)}
\newcommand{\imm}[1]{\operatorname{Imm}({#1}, \mathbb{R}^d)}
\newcommand{\embc}[1]{\operatorname{Emb}_c({#1}, \mathbb{R}^d)}
\newcommand{\immc}[1]{\operatorname{Imm}_c({#1}, \mathbb{R}^d)}
\newcommand{\embi}[1]{\operatorname{Emb}({#1}, \mathbb{R}^d)_{\iota}}
\newcommand{\immi}[1]{\operatorname{Imm}({#1}, \mathbb{R}^d)_{\iota}}
\newcommand{\embci}[1]{\operatorname{Emb}_c({#1}, \mathbb{R}^d)_{\iota}}
\newcommand{\immci}[1]{\operatorname{Imm}_c({#1}, \mathbb{R}^d)_{\iota}}
\newcommand{\oemb}[1]{\overline{\operatorname{Emb}}({#1}, \mathbb{R}^d)}
\newcommand{\oembc}[1]{\overline{\operatorname{Emb}}_c({#1}, \mathbb{R}^d)}
\newcommand{\oembi}[1]{\overline{\operatorname{Emb}}({#1}, \mathbb{R}^d)_{\iota}}
\newcommand{\oembci}[1]{\overline{\operatorname{Emb}}_c({#1}, \mathbb{R}^d)_{\iota}}
\newcommand{\cube}{\mathcal{P}(\omega)}
\newcommand{\SL}{\mathcal{SL}_{m,d}}
\newcommand{\oSL}{\overline{\mathcal{SL}}_{m,d}}
\newcommand{\iSL}{\mathcal{SL}^{\textnormal{imm}}_{m,d}}
\newcommand{\slgraph}{\mathcal{L}}
\newcommand{\mlgraph}{\mathcal{L}}
\newcommand{\ML}{\mathcal{ML}_{\iota}}
\newcommand{\oML}{\overline{\mathcal{ML}}_{\iota}}
\newcommand{\iML}{\mathcal{ML}^{\textnormal{imm}}_{\iota}}
\newcommand{\tML}{\widetilde{\mathcal{ML}}^{\textnormal{imm}}_{\iota}}
\newcommand{\Np}{\mathbb{N}_{>0}}
\newcommand{\R}{\mathbb{R}}
\newcommand{\A}{\mathcal{A}}
\newcommand{\C}{\mathcal{C}}
\newcommand{\Rm}{\mathbb{R}^{m}}

\begin{document}
\title{Representation stability of string links and manifold links}
\author{Filipp Buryak}
\address{Universit\'e Paris Cit\'e, Sorbonne Universit\'e, CNRS, IMJ-PRG, F-75013 Paris, France}
\begin{abstract} 
  We prove representation stability for rational cohomology of spaces of string links and manifold links in Euclidean space as the number of links grows. For string links we consider the component of the standard embedding of disjoint copies of $\R^{m}$ into $\R^{d}$, where $m \geq 1$ and $d \geq m+3$. For manifold links we consider the component determined by disjoint copies of a fixed embedding of a closed smooth $m$-manifold into $\R^{d-2}$, where $d \geq 3$. In both cases, the rational cohomology groups form finitely generated $\FIs$-modules in each degree with generation arity at most $n\frac{d-2}{d-m-2}$ in degree $n$. We also obtain explicit generation bounds for the rational homotopy groups. The proofs combine geometric constructions of structure maps, hairy graph complex models for embeddings modulo immersions and a functorial framework that allows generation bounds to be established after forgetting the symmetric group actions.
\end{abstract}
% \subjclass[2020]{57N35}
\maketitle

\section{Introduction}
Properties of configuration spaces are of interest as these spaces play a central role in various topics of algebraic topology with deep connections to other fields. A classical result in this area is the homological stability of unordered configuration spaces of connected open manifolds~\cite{mcduffConfigurationSpacesPositive1975,segalTopologySpacesRational1979}. It states that in any fixed degree the homology groups eventually stabilise as the number of points grows. For ordered configuration spaces $\textnormal{Conf}(n, M)$ the situation is different, the classical notion of homological stability fails. This can be explained by the fact that there is a canonical action of the symmetric groups $\Sigma_n$ on the spaces $\textnormal{Conf}(n, M)$, which we fail to account for. Studying this phenomenon Church, Farb and Ellenberg in \cite{Church2012}, \cite{CF2013}, \cite{CEF2015} developed an algebraic framework of $\FI$-modules which allows one to study stability of a sequence of $\Sigma_n$ representations. There they showed that for a manifold $M$ the spaces $\textnormal{Conf}(n, M)$ satisfy this new stability, known as representation stability, under some mild assumptions on $M$.

In this paper, we generalize the spaces in question by replacing points with manifolds. We are interested in the two following cases:
\begin{itemize}
  \item The spaces of string links $\embci{\Rm \times \underline{r}}$, that is, embeddings that coincide with a fixed linear embedding $\iota: \Rm \times \underline{r} \hookrightarrow \R^{m+1} \subset \R^d$ outside of a compact. Here we restrict our attention to the component containing $\iota$.
  \item The spaces of manifold links $\embi{M \times \underline{r}}$ where $M$ is a smooth closed $m$ dimensional manifold and $\iota$ a fixed embedding $M \hookrightarrow \R^{d-2}$. Again we restrict our attention to the component containing the embedding $\iota_{r}: M \times \underline{r} \hookrightarrow \R^{d-2} \subset \R^{d}$ obtained by stacking copies of $\iota$.
\end{itemize}
In both of these cases we show that the spaces in question assemble together into pointed homotopy $\FIs$-spaces, i.e. functors $\FIs \to \textnormal{Ho}(\textbf{Top}_{*})$, where $\FIs$ is the category of finite sets and partial injections. We call them $\SL$ and $\ML$ in the cases of spaces of string links and manifold links respectively. Passing to rational cohomology, we obtain two graded $\FIs$-modules  $H^*(\SL)$ and $H^*(\ML)$. The goal of this paper is to prove that these $\FIs$-modules are finitely generated and to find the explicit bounds for the arities in which they are generated. 

The strategy of the proof in both cases is the similar. We explain it in the case of manifold links since it is more involved. The first step is to construct the homotopy fiber sequence
\begin{equation}\label{eq:intro}
  \oembi{M \times \underline{r}} \to \embi{M \times \underline{r}} \to \widetilde{\operatorname{Imm}}(M \times \underline{r}, \R^d)_{\iota},
\end{equation}
where $\widetilde{\operatorname{Imm}}(M \times \underline{r}, \R^d)_{\iota}$ is a suitably chosen covering space of $\immi{M \times \underline{r}}$.
Further, we ask this sequence to be natural meaning that there is a sequence of homotopy $\FIs$-spaces
\[
\oML \to \ML \to \tML,
\]
which evaluates to \eqref{eq:intro} on the set $\underline{r}$ for any $r \in \mathbb{N}$.
 
The second step is to use the rational models of Fresse, Turchin, Willwacher \cite{FTW2020} describing the rational homotopy type of the spaces of embeddings modulo immersions $\oemb{M}$ in terms of graph complexes. By analyzing the behavior of these graph complexes we prove the statement of the theorem for the graded $\FIs$-module $H^{*}(\oML; \field)$.

The last step is to use Serre spectral sequence for rational cohomology with local coefficients. By carefully analyzing the covering spaces $\widetilde{\operatorname{Imm}}(M \times \underline{r}, \R^d)_{\iota}$ and, in particular, the action of their fundamental groups on the cohomology of the fiber, we are able to deduce the result. The two main theorems that we obtain are the following.
\begin{theorem}[\cref{thm:cohomology_ml}, \cref{prop:bound_improvement_ml}]
  Let $d \geq 3$ and let $\iota$ be an embedding of an $m$ dimensional closed manifold $M$ into $\R^{d-2}$. The graded $\FIs$-module $H^*(\ML; \field)$ is of finite type with slope  $\leq \frac{d-2}{d-m-2}$. If $d \geq 2m+2$ then the slope can be improved to $\leq \frac{2}{d-2m-1}$.
\end{theorem}
\begin{theorem}[\cref{thm:cohomology_sl}, \cref{prop:bound_improvement_sl}]
  Let $m \geq 1$ and $d \geq m+3$. The graded $\FIs$-module $H^*(\SL; \field)$ is of finite type with slope $\leq \frac{d-2}{d-m-2}$.
\end{theorem}

\subsection{Existing literature}
Representation stability for ordered configuration spaces was established by Church and Farb in \cite{Church2012}, \cite{CF2013}. The $\FIs$-module reformulation were developed by Church, Ellenberg, Farb and Nagpal in \cite{CEF2015}, \cite{CEFN2014}. In \cite{HR2017} Hersh and Reiner determine sharp representation-stability ranges for Euclidean configuration spaces.

For positive-dimensional submanifolds, Wilson’s \cite{Wilson2012} showed representation stability for the cohomology of pure string motion groups, Kupers’s \cite{Kupers2020} showed homological stability for unlinked circles in $3$-manifolds, and Cantero Mor{\'a}n and Randal-Williams’s \cite{CMRW2017} showed stability for embedded surfaces as the genus increases. 

Later, Palmer \cite{Palmer2021} proves homological stability for components of moduli spaces in the case of open ambient manifolds of dimension $d \geq 2m+3$ and deduces representation stability for their ordered counterparts. Gu{\`e}s \cite{Gues2025} improves the range to $d \geq 2m+2$, obtaining stability results for both open and closed ambient manifolds. Our results give explicit $\FIs$ generation bounds for rational cohomology and homotopy of chosen components of parametrised manifold-link and string-link spaces, including cases outside the dimension ranges of the theorems for moduli spaces.

\subsection{Structure of the paper.}
In section 1, we introduce the language of functor categories, focusing on $\FI$-modules. There we mostly discuss $\FIs$-modules, also known as free $\FI$-modules. The main reason to do so is that all $\FI$-modules appearing in this paper turn out to be free. Similarly to the configuration spaces, this happens because $\R^d$ can be viewed as an interior of a compact manifold with non-empty boundary. We also introduce a new category $\cube$ modeling an infinite dimensional cubic diagram. One can think about this category as an ``untangling'' of $\FI$. It gives a technical advantage by allowing us to analyze $\FI$-modules after discarding the actions of the symmetric groups.

In section 2 we quickly review the results of embedding calculus allowing one to study $\oembc{M \times \underline{r}}$ via operad theory. After that we briefly discuss the Hairy Graph Complex construction and prove the lemmas which are important later. We finish the section by recalling the main results of \cite{FTW2020} which, in particular, describe the rational homotopy type of $\oembc{M \times \underline{r}}$ in terms of Hairy Graph Complexes.

Section 3 is devoted to the study of string links. A big part of this section is \cref{Structure maps} where we construct the pointed homotopy $\FIs$-space $\SL$. It gives a geometrical explanation as to why the sequence of homotopy groups of string links assembles into an $\FIs$-module, which could be easily seen on the level of Hairy Graph Complexes. The rest of the section is devoted to showing that rational algebraic invariants of $\SL$ satisfy representation stability. It is done by showing that the $\FIs$ space of the homotopy fibers $\oSL$ satisfies representation stability and using it to deduce the results for $\SL$.

Section 4 is devoted to manifold links. It is very similar in spirit to section 3. However, changes do appear in the later parts of the section due to the fact that the spaces of manifold links are not loop spaces.

\subsection{Acknowledgments}
First of all, I would like to thank my advisiors Najib Idrissi and Muriel Livernet for many fruitful discussions and comments. I would also like to thank Nathalie Wahl and Christine Vespa for useful discussions. This work was supported by ANR grant ANR-22-CE40-0008 SHoCoS.

\section{Functor categories}
In this section we briefly discuss functor categories shaped on various categories of finite sets. Of interest will be the following categories:
\begin{itemize}
  \item $\FB$ the category of finite subsets of $\Np$ and bijections;
  \item $\FI$ the category of finite subsets of $\Np$ and injections;
  \item $\FIs$ the category of finite subsets of $\Np$ and partial injections;
  \item $\cube$ the category of finite subsets of $\Np$ and inclusions.
\end{itemize}
When it is not important which exactly category of the above we are talking about we will use the notation $\mathcal{I}$ to refer to any of them.
\begin{remark}
  Typically, $\FB, \FI$ and $\FIs$ denotes the corresponding categories shaped on all finite sets. However, restricting our attention just to finite subsets of $\Np$ produces equivalent categories.
\end{remark}
The categories $\FB, \FI$ and $\FIs$ admit small skeletons. They are given by restricting our attention to the sets $\underline{n}=\{1,2,\dots,n\}$ for $n\in\mathbb{N}$. In particular, we set $\underline{0}=\emptyset$.

We work over $\field$. The target category $\C$ of our functors will typically be (graded, differential graded) $\field$ vector spaces or algebras. An element $V$ of the functor category $\textnormal{Fun}(\mathcal{I}, \C)$ is typically called (graded, differential graded) $\mathcal{I}$-module or -algebra. For a finite set $S$ we often write $V_S$ for $V(S)$. Similarly, we often write $V_n$ for $V(\underline{n})$. We also call $V_n$ the component of arity $n$ of the $\mathcal{I}$ module. We write $\Res^{\mathcal{J}}_{\mathcal{I}}$ for the restriction functor $\textnormal{Fun}(\mathcal{J}, \C) \to \textnormal{Fun}(\mathcal{I}, \C)$ associated to two categories $\mathcal{I} \subset \mathcal{J}$. We often write $\Res$ for $\Res^{\mathcal{J}}_{\mathcal{I}}$ when it is clear from context what categories we are talking about.

\subsection{FB-modules}
The data of an $\FB$-module $V$ in $\C$ is equivalent to a sequence of representations $V_n$ of the symmetric groups $\Sigma_n$ in $\C$.
\begin{definition}
  We say that an $\FB$-module is
  \begin{itemize}
    \item finitely generated if all $V_n$ are finitely dimensional vector spaces and only a finite number of them are non-zero;
    \item generated in arity $\leq n$ if $V_m=0$ for $m > n$;
    \item finitely generated in arity $\leq n$ if it is finitely generated and generated in arity $\leq n$.
  \end{itemize}
\end{definition}
Since the symmetric groups are isomorphic to their opposites the category $\FB$ is self-opposite. More explicitly, there is an isomorphism $\eta: \FB \to \FB^{\op}$ which sends the two sets $S$ and $T$ and a map between them $f:S\to T$ to $f^{-1}: T \to S$. By abuse of notation, we identify $\FB$ and $\FB^{\op}$ and regard $\eta$ as an involution of $\FB$. This induces an involution of $\textnormal{Fun}(\FB, \C)$ given by precomposing with $\eta$.
\begin{remark}
  Let $X$ be an $\FB$ object in a category $\C$ and let $F:\C \to \mathcal{D}$ be a contravariant functor. By abuse of notation, from now on we write $F(X)$ to denote the $\FB$ object $F(X(\eta))$ in $\mathcal{D}$.
\end{remark}

\subsection{FI-modules}
For proper introduction to $\FI$-modules consult \cite{CEF2015,CE2017}.
\begin{definition}
  The $\FI$ homology functor $\mathbf{H}: \FI\textnormal{-Mod} \to \FB\textnormal{-Mod}$ is defined on an $\FI$-module $V$ as
  \[
  \mathbf{H}(V)_T = \operatorname{coker} (\bigoplus_{S\varsubsetneq T} V_S \to V_T ).
  \]
  We say that an $\FI$-module is finitely generated (resp., generated in arity $\leq n$, finitely generated in arity $\leq n$) if the $\FB$-module $\mathbf{H}(V)$ is finitely generated (resp., generated in arity $\leq n$, finitely generated in arity $\leq n$).
\end{definition}
\begin{definition}
  We say that a finite (possibly empty) non-increasing sequence of positive integers $\lambda=(\lambda_1\geq\lambda_2\geq\dots\geq\lambda_l>0)$ is a partition of $|\lambda|:=\lambda_1+\lambda_2+\dots+\lambda_l$. For $n\geq |\lambda|+\lambda_1$ we define the padded partition
  \[
  \lambda[n]:=(n-|\lambda|,\lambda_1,\lambda_2,\dots,\lambda_l).
  \]
  We also define $V(\lambda)_n$ to be the irreducible representation of $\Sigma_n$ corresponding to the partition $\lambda[n]$.
\end{definition}
Every partition of $n$ can be written as $\lambda[n]$ for a unique partition $\lambda$. Since the irreducible representations of $\Sigma_n$ over $\field$ are classified by the partitions of $n$, any irreducible representation of $\Sigma_n$ is isomorphic to $V(\lambda)_n$ for some $\lambda$.

One of the reasons to introduce $\FI$-modules in \cite{CEF2015} was to study the consistent sequences of \cite{CF2013}.
\begin{definition}
  A consistent sequence is a sequence $\{V_n\}_{n\in\mathbb{N}}$ of $\Sigma_n$ representations together with linear maps $\phi_n:V_n\to V_{n+1}$. The maps are required to be equivariant in the sense that the following square commutes
  \[
  \begin{tikzcd}
    V_n \arrow{r}{\phi_n} \arrow[swap]{d}{g} & V_{n+1} \arrow{d}{g\times id} \\
    V_n \arrow{r}{\phi_n} & V_{n+1}
  \end{tikzcd}
  \]
  for every $g\in\Sigma_n$.
\end{definition}
\begin{lemma}[{{\cite[Remark 3.3.1]{CEF2015}}}]
  Every $\FI$-module $V$ gives rise to a consistent sequence by setting $V_n := V(\underline{n})$ and $\phi_n := V(\underline{n}\hookrightarrow\underline{n+1})$. A consistent sequence $\{V_n\}$ can be promoted to an $\FI$-module if and only if the transposition $(n+1 \; n+2)\in\Sigma_{n+2}$ acts trivially on $\im(\phi_{n+1}\phi_n)$ for every $n\in\mathbb{N}$.
\end{lemma}
\begin{definition}
  A consistent sequence $\{(V_n,\phi_n)\}_{n\in\mathbb{N}}$ is representation stable with stable range $N\in\mathbb{N}$ if, for all $n\geq N$, the following holds.
  \begin{itemize}
    \item The map $\phi_n:V_n\to V_{n+1}$ is injective for $n\geq N$.
    \item The span $\langle\sigma\phi_n(V_n)|\sigma\in\Sigma_{n+1}\rangle$ is equal to $V_{n+1}$ for $n\geq N$.
    \item Decompose $V_n$ into irreducible representations as
    \[
    V_n = \bigoplus_\lambda c_{\lambda,n} V(\lambda)_n
    \]
    with multiplicities $0\leq c_{\lambda,n}\leq \infty$. For each $\lambda$ the multiplicities $c_{\lambda, n}$ are independent of $n$ for $n\geq N$.
  \end{itemize}
\end{definition}
\begin{proposition}[{{\cite[Theorem 1.13]{CEF2015}}}]
  An $\FI$-module $V$ is finitely generated if and only if the underlying consistent sequence is representation stable and all $V_n$ are finite dimensional.
\end{proposition}

\subsection{$\FIs$-modules}
Recall that $\FIs$ is the category of finite subsets of $\Np$ and partial injections. A partial injection is a map $f:S \to T$ defined only on a subset $\operatorname{dom} f \subset S$ on which it is injective. There are three special types of partial injections
\begin{itemize}
  \item subset inclusions $\phi^{T}_{S}: S \hookrightarrow T$ for $S \subset T$;
  \item projections on a subset $\psi^{T}_{S}: T \dashrightarrow S$ for $S \subset T$;
  \item bijections $\sigma: S \xrightarrow{\sim} S'$.
\end{itemize}
Every partial injection $f: S \to T$ can be decomposed as the following composition
\[
\begin{tikzcd}
  S \arrow[r, "f"] \arrow[d, dashed] & T \\
  \operatorname{dom} f \arrow[r, "\sim"] & \im f \arrow[u, hook]
\end{tikzcd}
\]
where the bottom map is induced from $f$.

Similarly to $\FB$, the category $\FIs$ is self-opposite via the map $\eta$ sending a partial injection $f:S \to T$ to its ``inverse'' $\eta(f) : T \to S$, defined by:
\[
\eta(f)(t) =
\begin{cases}
  s, & \text{if } t = f(s) \text{ for some } s \in S, \\
  \text{undefined}, & t \notin \textnormal{im}(f).
\end{cases}
\]
By abuse of notation, we identify $\FIs$ and $(\FIs)^{\op}$ and regard $\eta$ as an involution of $\FIs$. Hence, precomposing with $\eta$ defines a natural involution of $\FIs$-modules.
\begin{remark}
  Let $X$ be an $\FIs$ object in a category $\C$ and let $F:\C \to \mathcal{D}$ be a contravariant functor. By abuse of notation, from now on we write $F(X)$ to denote the $\FIs$ object $F(X(\eta))$ in $\mathcal{D}$.
\end{remark}

\begin{definition}
  We define two different functors $\crosseffect, \; \mathbf{H}: \FIs\textnormal{-Mod} \to \FB\textnormal{-Mod}$ by setting
  \begin{gather*}
    \crosseffect(V)_T = \operatorname{ker}(V_T \to \bigoplus_{S\varsubsetneq T}V_S), \\
    \mathbf{H}(V)_T = \operatorname{coker}(\bigoplus_{S\varsubsetneq T}V_S \to V_T).
  \end{gather*}
  There is an obvious natural transformation $\nu:\crosseffect \Rightarrow \mathbf{H}$ given on object $T$ by composition $\crosseffect(V)_T \to V_{T} \to \mathbf{H}(V)_T$.
\end{definition}
\begin{proposition}\label{injective and projective homology}
  The natural transformation $\nu:\crosseffect \Rightarrow \mathbf{H}$ is an isomorphism.
\end{proposition}
In order to prove the Proposition we introduce the following technical tool. For $S\subset T$ let $\textnormal{pr}_S \in \Hom_{\FIs}(T,T)$ given by $\operatorname{pr}_S(s) = s$ for all $s \in S$ (and undefined outside $S$). For an $\FIs$-module $V$ and $x\in V_T$ we write $x_S$ to denote $V(\textnormal{pr}_S)(x)$. We set
\[
\operatorname{tl} : V_T \to V_T, \qquad \operatorname{tl}(x) = \sum_{i=0}^{|T|} \sum_{\substack{S\subset T \\ |S| = i}} (-1)^{|T|-i} x_S.
\]
\begin{lemma}\label{lem:tl}
  Let T be a finite set, $V$ be an $\FIs$-module and $x \in V_T$. The following holds
  \begin{enumerate}
    \item $(x_S)_{S'} = x_{S\cap S'} = (x_{S'})_S$ for any $S, S'\subset T$.
    \item $\operatorname{tl}(x) \in \crosseffect(V)_T$ for any $x \in V_T$.
    Furthermore, if $x\in \crosseffect(V)_T$ then $x=\operatorname{tl}(x)$.
    \item $[x]=[\operatorname{tl}(x)]$ in $\mathbf{H}(V)_T$.
    Furthermore, $\textnormal{tl}(x) = 0$ for any $x \in \textnormal{im}\bigl(\bigoplus_{S\varsubsetneq T}V_S \to V_T\bigr)$.
  \end{enumerate}
\end{lemma}
\begin{proof}
  \textbf{1.} Note that, as partial injections, $\textnormal{pr}_S\circ \textnormal{pr}_{S'} = \textnormal{pr}_{S\cap S'} = \textnormal{pr}_{S'}\circ \textnormal{pr}_S$.

  \textbf{2.} Since the subset inclusion $T\setminus\{t\} \hookrightarrow T$ admits a left inverse in $\FIs$
  we have that $\text{ker}\bigl(V_T\to V_{T\setminus\{t\}}\bigr) = \text{ker}\bigl(V(\textnormal{pr}_{T\setminus\{t\}})\bigr)$. Observe that for $x \in V_T$
  \begin{align*}
    \operatorname{tl}(x)_{T\setminus\{t\}} &= \sum_{i=0}^{|T|} \sum_{\substack{S\subset T \\ |S| = i}} (-1)^{|T|-i} x_{S\cap T\setminus\{t\}}  \\
    &=
    \sum_{i=0}^{|T|-1} \sum_{\substack{S\subset T\setminus\{t\} \\ |S| = i}} (-1)^{|T|-i} x_{S} +
    \sum_{i=1}^{|T|} \sum_{\substack{t\in S\subset T \\ |S| = i}} (-1)^{|T|-i} x_{S\setminus\{t\}} \\
    &=
    \sum_{i=0}^{|T|-1} \sum_{\substack{S\subset T\setminus\{t\} \\ |S| = i}} (-1)^{|T|-i} x_{S} +
    \sum_{i=0}^{|T|-1} \sum_{\substack{S\subset T\setminus\{t\} \\ |S| = i}} (-1)^{|T|-i-1} x_{S} = 0.
  \end{align*}
  For any $S\varsubsetneq T$, the retraction $T \twoheadrightarrow S$ factors through a retraction $T \twoheadrightarrow T \setminus \{t\}$ for some $t \in T$. Therefore, $\text{tl}(x) \in \text{ker}\bigl(V_T \to V_S\bigr)$ for all $S \varsubsetneq T$. Suppose $x \in \crosseffect(V)_T$ that is $x \in \text{ker}\bigl(V_T \to \bigoplus_{S\varsubsetneq T}V_S\bigr)$. Then $x_S = 0$ for any $S\varsubsetneq T$ since $\textnormal{pr}_S$ factors through the retraction $T\twoheadrightarrow S$. Hence, $\text{tl}(x) = x_T$ and by definition $x_T = x$.

  \textbf{3.} The retraction $T \twoheadrightarrow S$ admits a right inverse in $\FIs$ so $\textnormal{im}\bigl(V_S \to V_T\bigr) = \textnormal{im}(V(\textnormal{pr}_S))$.
  Hence, $[x_S] = 0$ in $\mathbf{H}(V)_T$ for $S\varsubsetneq T$ and $[\text{tl}(x)]=[x_T]=[x]$. For the second part, since $\textnormal{im}\bigl(\bigoplus_{S\varsubsetneq T}V_S \to V_T\bigr) = \sum_{S\varsubsetneq T} \textnormal{im}\bigl(V_S \to V_T\bigr)$ it is enough to prove that $\textnormal{tl}(x) = 0$ for any $x \in \textnormal{im}\bigl(V_S \to V_T\bigr)$ and for any $S \varsubsetneq T$. By 1. $\textnormal{tl}(x_S) = \textnormal{tl}(x)_S$ and in 2. we showed that the latter is equal to $0$ for $S \varsubsetneq T$.
\end{proof}
\begin{proof}[Proof of Proposition \ref{injective and projective homology}]
  The result follows from 2. and 3. of \cref{lem:tl}.
\end{proof}
\begin{corollary}\label{Exactness of functor homology}
  The functors $\crosseffect$ and $\mathbf{H}$ are exact.
\end{corollary}
\begin{proof}
  Follows from the fact that $\crosseffect$ is left exact, $\mathbf{H}$ is right exact and \cref{injective and projective homology}
\end{proof}

\begin{definition}
  We say that an $\FIs$-module $V$ is finitely generated (resp., generated in arity $\leq n$, finitely generated in arity $\leq n$) if one of the following equivalent conditions hold:
  \begin{itemize}
    \item the $\FB$-module $\crosseffect(V)$ is finitely generated (resp., generated in arity $\leq n$, finitely generated in arity $\leq n$);
    \item the $\FB$-module $\mathbf{H}(V)$ is finitely generated (resp., generated in arity $\leq n$, finitely generated in arity $\leq n$);
    \item the underlying $\FI$-module $\Res(V)$ is finitely generated (resp., generated in arity $\leq n$, finitely generated in arity $\leq n$).
  \end{itemize}
\end{definition}
\begin{lemma}\label{lem:thick}
  The full subcategory $\mathcal{I}$ of finitely generated (generated in arity $\leq n$, finitely generated in arity $\leq n$) $\FIs$-modules is thick, i.e. for every exact sequence
  \[
  0 \to U \to V \to W \to 0
  \]
  $V$ is an object in $\mathcal{I}$ if and only if $U$ and $W$ are objects in $\mathcal{I}$.
\end{lemma}
\begin{definition}
  Let $V$ be an $\FB$-module or an $\FIs$-module. We define its dual $\FIs$-module $V^\vee$ as
  \[
  V^\vee = \Hom(-, \field) \circ V \circ \eta.
  \]
  Explicitly, for a partial injection $f$ we have
  \[
  V^\vee(f) = \Hom(V(\eta(f)), \field).
  \]
\end{definition}
\begin{proposition}\label{prop:properties_dual}
  An $\FIs$-module $V$ is finitely generated (resp., generated in arity $\leq n$) if and only if $V^\vee$ is finitely generated (resp., generated in arity $\leq n$).
\end{proposition}
\begin{proof}
  Follows from the observation that
  \[
  \crosseffect(V^{\vee})(S) \simeq \Hom(\mathbf{H}(V)(S), \field). \qedhere
  \]
\end{proof}
\begin{definition}
  We define an induction functor $\Ind: \FB\textnormal{-Mod} \to \FIs\textnormal{-Mod}$ on an $\FB$-module $V$ as
  \[
  \Ind(V)_{T} = \bigoplus_{S\subset T} V_{S},
  \]
  where the projections act by projecting onto the corresponding direct summands, the subset inclusions act by inclusion of the corresponding direct summands and bijections act by permuting the direct summands.
\end{definition}
\begin{proposition}[{{\cite[Theorem 4.1.5]{CEF2015}, \cite{Slominska2004}, \cite{Pirashvili2000}}}]\label{Equivalence of FB and FI_sharp modules}
  The category of $\FIs$-modules is equivalent to the category of $\FB$-modules, via the equivalence of categories
  \[
  \Ind: \textnormal{FB-Mod} \rightleftarrows \FIs\textnormal{-Mod}: \mathbf{H}.
  \]
\end{proposition}
\begin{corollary}
  Let $V$ be an $\FIs$-module generated in arity $\leq n$. Then
  \[
  V_{T} \simeq \Ind(\crosseffect(V))_T = \bigoplus_{\substack{S \subset T \\ \lvert S \rvert \leq n}} \crosseffect (V)_{S}.
  \]
\end{corollary}
\begin{proposition}[{{\cite[Corollary 4.1.8]{CEF2015}}}]\label{Bounds for representation stablity}
  Let $V$ be an $\FIs$-module. If $V$ is finitely generated in arity $\leq n$ then the underlying consistent sequence is representation stable with stable range $2n$.
\end{proposition}

\subsection{$\cube$-modules}
Recall that $\cube$ is the category of finite subsets of $\Np$ and the subset inclusions. In this paper we are interested in $\cube^{\op}$-modules. Note that the core of $\cube$ is a discrete category whose object are finite subsets of $\Np$. We denote it $\mathcal{P}^{\simeq}$.
\begin{definition}
  Let $V$ be a $\cube$-module. We define the homology functor $\mathbf{H}: \cube\textnormal{-Mod} \to \mathcal{P}^{\simeq}\textnormal{-Mod}$ as
  \[
  \mathbf{H}(V)_T = \operatorname{coker}(\bigoplus_{S \subsetneq T}V_{S} \to V_{T}).
  \]
  We say that $V$ is generated in arity $\leq n$ if $\mathbf{H}(V)$ vanishes on all sets with cardinality $> n$. Moreover, we say that $V$ is finitely generated if it is arity-wise finite-dimensional and generated in arity $\leq n$ for some $n\in\mathbb{N}$.
\end{definition}
\begin{definition}
  Let $V$ be a $\cube^\op$-module. We define the cross effect functor $\crosseffect: \cube^{\op}\textnormal{-Mod} \to \mathcal{P}^{\simeq}\textnormal{-Mod}$ as
  \[
  \crosseffect(V)_T = \operatorname{ker}(V_T \to \bigoplus_{S \subsetneq T}V_{S}).
  \]
  We say that $V$ is cogenerated in arity $\leq n$ if $\crosseffect(V)$ vanishes on all sets with cardinality $> n$. Moreover, we say that $V$ is finitely cogenerated if it is arity-wise finite-dimensional and cogenerated in arity $\leq n$ for some $n\in\mathbb{N}$.
\end{definition}
The following result is almost tautological.
\begin{lemma}\label{lem:properties_cube}
  An $\FIs$-module $V$ is generated in arity $\leq n$ (resp., finitely generated) if and only if one of the two equivalent conditions hold:
  \begin{itemize}
    \item the underlying $\cube$-module $\Res(V)$ is generated in arity $\leq n$ (resp., finitely generated);
    \item the underlying $\cube^\op$-module $\Res(V)$ is cogenerated in arity $\leq n$ (resp., finitely cogenerated).
  \end{itemize}
\end{lemma}
\begin{proof}
  Follows from the commutativity of the following diagrams.
  \[
  \begin{tikzcd}
    \FIs\text{-Mod} \arrow[r, "\Res"] \arrow[d, "\mathbf{H}"] & \cube\textnormal{-Mod} \arrow[d, "\mathbf{H}"] &
    \FIs\text{-Mod} \arrow[r, "\Res"] \arrow[d, "\crosseffect"] & \cube^{\op}\textnormal{-Mod} \arrow[d, "\crosseffect"] \\
    \FB\text{-Mod} \arrow[r, "\Res"] & \mathcal{P}^{\simeq}\text{-Mod} &
    \FB\text{-Mod} \arrow[r, "\Res"] & \mathcal{P}^{\simeq}\text{-Mod}
  \end{tikzcd} 
  \]
\end{proof}

\subsection{(Differential) graded $\FB, \FI$ and $\FIs$-modules and algebras}
We write $\mathcal{I}$ to denote any of the categories $\FB, \FI$ or $\FIs$. Often the target category of our functors is not just $\field$ vector spaces but the category of (differential) graded $\field$ vector spaces/algebras. We call such functors (differential) graded $\mathcal{I}$-modules/algebras.

The data of a graded $\mathcal{I}$-module $V$ is equivalent to the data of a sequence $\{V^t\}_{t\in\mathbb{Z}}$ of $\mathcal{I}$-modules. The notion of being finitely generated is often not satisfied by the graded modules of interest. The problem is that a finitely generated graded $\mathcal{I}$-module has to be concentrated in a finite number of degrees. Instead, we use the notion of being of finite type.
\begin{definition}
  Let $V$ be a (differential) graded $\mathcal{I}$-module. We say that $V$ is of finite type if $V^t$ is finitely generated for all $t\in\mathbb{Z}$. If $\mathcal{I} = \FIs$ we say that $V$ has slope $\leq n$ if $(V^t)$ is generated in arity $\leq nt$ for all $t\in\mathbb{Z}$.
\end{definition}
\begin{remark}
  Note that in \cite{CEF2015} slope is defined differently. However, it follows from \cite[Proposition 3.2.4]{CEF2015} that two definitions coincide for $\FIs$-modules over $\mathbb{Q}$.
\end{remark}
\begin{proposition}[{{\cite[Theorem 4.2.3]{CEF2015}}}]\label{prop:algebra_slope}
  Let $A$ be a graded $\FI$-algebra generated by a graded $\FI$-submodule $V$ with $V^t=0$ for $t\leq 0$. If $V$ is of finite type, then $A$ is of finite type. If $V$ has slope $\leq n$, then $A$ has slope $\leq n$.
\end{proposition}

The construction of the induction functor from $\FB$-modules to $\FIs$-modules works in the (differential) graded context without change.
\begin{lemma}\label{homology of the free functors}
  Let $V$ be a dg $\FB$-module. The graded $\FIs$-modules $\Ind (H(V))$ and $H(\Ind(V))$ are naturally isomorphic.
\end{lemma}
\begin{proof}
  The induced module $\Ind(V)_{T}$ is just a direct sum of dg vector spaces $V_{S}$ for $S \subset T$. Hence,
  \[
  H(\Ind(V)_{T}) = H(\bigoplus_{S \subset T} V_{S}) \simeq \bigoplus_{S \subset T} H(V_{S}) = \Ind(H(V))_T,
  \]
  which is clearly natural.
\end{proof}
\begin{lemma}
  Let $\mathcal{O}$ be a reduced operad in (differential graded) $\field$ vector spaces and let $A$ be an $\FB$ $\mathcal{O}$-algebra, i.e. a functor from $\FB$ to the category of $\mathcal{O}$-algebras. Then $\Ind(A)$ is an $\FIs$ $\mathcal{O}$-algebra.
\end{lemma}
\begin{proof}
  Since $\mathcal{O}$ is a reduced operad all structure maps of $\Ind(A)$ are $\mathcal{O}$-algebra maps.
\end{proof}
\begin{corollary}\label{lem:ind_of_algebra}
  Let $\mathcal{O}$ be a reduced operad in (differential graded) $\field$ vector spaces and let $A$ be an $\mathcal{O}$-algebra. Then $\Ind(A)$, where we view $A$ as an $\FB$ $\mathcal{O}$-algebra concentrated in arity $1$, is an $\FIs$ $\mathcal{O}$-algebra such that
  \[
  \Ind(A)_{S} = A^{\times S},
  \]
  where the projections act by projecting onto the corresponding direct summands, the subset inclusions act by inclusion of the corresponding direct summands and bijections act by permuting the direct summands.
\end{corollary}

\section{Recollection on the spaces of embeddings modulo immersions}
\subsection{Setting}\label{subsection:setting}
Let $\R^d$ be the Euclidean space and $\R^{m+1} = \R^{m+1} \times \{0\}^{d-m-1}$ be a linear subspace, where $d \geq m+1$. Let $U$ be a subspace of $\R^{m+1}$. We write $\iota_U: U \to \R^d$ to denote the standard inclusion $U \subset \R^{m+1} \subset \R^d$. Let $\text{Emb}_c(U, \R^d)$, respectively $\text{Imm}_c(U, \R^d)$, be the space of smooth embeddings, respectively immersions, which coincide with $\iota_U$ outside a compact. We assume that $\text{Emb}_c(U, \R^d)$ and $\text{Imm}_c(U, \R^d)$ are pointed at $\iota_U$. There is an obvious inclusion:
\[
\text{Emb}_c(U, \R^d) \hookrightarrow \text{Imm}_c(U, \R^d).
\]
Let $\overline{\text{Emb}}_c(U, \R^d)$ be the homotopy fiber of this inclusion. We fix a model for this homotopy fiber by setting
\[
\overline{\text{Emb}}_c(U, \R^d) = \{\gamma:[0,1]\to \text{Imm}_c(U, \R^d) \mid \gamma(0)\in \text{Emb}_c(U, \R^d), \gamma(1)=\iota_U\}.
\]

We write $\overline{\textnormal{Emb}}_c(U, \R^d)_\iota$ to denote the connected component containing $\gamma_U: [0,1] \to \text{Imm}_c(U, \mathbb{R^d})$ where $\gamma_U(t) = \iota_U$ for all $t\in [0,1]$. Furthermore, we assume that $\overline{\textnormal{Emb}}_c(U, \R^d)$ and $\overline{\textnormal{Emb}}_c(U, \R^d)_\iota$ are pointed at $\gamma_U$.

\begin{definition}
  Let $\widetilde{\mathcal{O}}_c(\R^{m+1})$ be the poset of open subsets of $\R^{m+1}$ whose complement is a compact non-empty manifold. The constructions above define the functors
  \begin{gather*}
    \overline{\textnormal{Emb}}_c(-, \R^d): \widetilde{\mathcal{O}}_c(\R^{m+1})^\op \to \textnormal{Top}_*; \\
    \overline{\textnormal{Emb}}_c(-, \R^d)_{\iota}: \widetilde{\mathcal{O}}_c(\R^{m+1})^\op \to \textnormal{Top}_*.
  \end{gather*}
\end{definition}

\subsection{Fulton-MacPherson operad}

Let $\mathcal{F}_{m+1}$ be the $(m+1)$-dimensional Fulton--MacPherson operad introduced by \cite[Section 3]{GetzlerJones1994}. There is a similarly defined infinitesimal $\mathcal{F}_{m+1}$-bimodule $\mathcal{IF}_{m+1}$ described in \cite[Section 6]{Turchin2013}. Its variation $\mathcal{IF}_U$ given in \cite{Turchin2013} and \cite[Section 2]{FTW2020} will be of special interest.

We will briefly describe $\mathcal{IF}_U$ here. Let $U$ be an object of $\widetilde{\mathcal{O}}(\R^{m+1})$ and $S$ a finite set. The points of the topological spaces $\mathcal{IF}_U(S)$ are represented by finite rooted trees whose leaves are indexed by elements of $S$ and which have one distinguished vertex. Each non-distinguished vertex is required to have at least $2$ children while the distinguished vertex can have any amount of children. Further, all vertices of such trees are decorated. A non-distinguished vertex with $k$ children is decorated by a quotient of the configuration space $\text{Conf}(k, \R^{m+1})$ by translations and pointwise scalar multiplications. The distinguished vertex with $k$ children is decorated by the configuration space $\text{Conf}(k, U)$.

We write $\mathcal{IF}_{m+1}$ for the special case where $U = \R^{m+1}$. Geometrically, the space $\mathcal{IF}_{m+1}(S)$ is a variant of Fulton--MacPherson--Axelrod--Singer compactification of a configuration space $\text{Conf}(S\cup \{*\}, \mathbb{S}^{m+1})$, where $\mathbb{S}^{m+1} = \R^{m+1} \cup \{\infty\}$ is the one point compactification of $\R^{m+1}$ and a point labeled by $*$ is required to be $\infty$. Similarly, $\mathcal{IF}_U(S)$ is the preimage $\pi^{-1}(U_*^S)$, where $U_* = U\cup\{\infty\}$ and $\pi$ is the obvious projection from $\mathcal{IF}_{m+1}(S)$ onto ${(\mathbb{S}^{m+1})}^S$.

For every object $U$ of $\widetilde{\mathcal{O}}_c(\R^{m+1})$, the spaces $\mathcal{IF}_U(S)$ assemble together into infinitesimal $\mathcal{F}_{m+1}$-bimodules, where both left and right actions are given by the grafting of trees.

Note that the inclusion of $\R^{m+1}$ into $\R^d$ defines an inclusion of operads $\mathcal{F}_{m+1} \xhookrightarrow{} \mathcal{F}_d$. This defines a canonical structure of an infinitesimal $\mathcal{F}_{m+1}$-bimodule on $\mathcal{IF}_d := \Res_{\mathcal{F}_{m+1}}^{\mathcal{F}_d}\mathcal{IF}_d$.

\begin{proposition}
  Let $d \geq m+3$. The natural map
  \[
  \overline{\textnormal{Emb}}_c(U, \R^d) \to \textnormal{IBimod}^h_{\mathcal{F}_{m+1}}(\mathcal{IF}_U, \mathcal{IF}_d^\field),
  \]
  induces a rational equivalence of nilpotent spaces on each connected component to its image, and is finite-to-one at the $\pi_0$ level. Here the domain of the functors is $\widetilde{\mathcal{O}}_c(\R^{m+1})^\op$.
\end{proposition}

\begin{proof}
  Let $\widetilde{\mathcal{O}}(\R^{m+1})$ denote the poset of open subsets of $\R^{m+1}$ containing a neighbourhood of $\infty$. It was shown in \cite[Section 6]{Turchin2013} that for any $k \in \mathbb{N}\cup\{\infty\}$ there is a natural transformation
  \begin{equation}\label{Taylor approximation}
    \overline{\textnormal{Emb}}_c(U, \R^d) \to \textnormal{IBimod}^h_{\mathcal{F}_{m+1},\leq k}(\mathcal{IF}_U, \mathcal{IF}_d)
  \end{equation}
  expressing the right-hand side as $k$-th Taylor approximation in the sense of Goodwillie-Weiss manifold calculus. In particular, the functors here have domain $\widetilde{\mathcal{O}}(\R^{m+1})^\op$. The convergence result of \cite{GK2015,GKW2001,GW1999} imply that for $k=\infty$ in \eqref{Taylor approximation} the natural transformation evaluated on $U\neq\R^{m+1}$ is a weak-equivalence, when $d \geq m+3$. Therefore, there is a natural weak equivalence
  \[
  \overline{\textnormal{Emb}}_c(U, \R^d) \to \textnormal{IBimod}^h_{\mathcal{F}_{m+1},\leq \infty}(\mathcal{IF}_U, \mathcal{IF}_d) \coloneq \textnormal{IBimod}^h_{\mathcal{F}_{m+1}}(\mathcal{IF}_U, \mathcal{IF}_d)
  \]
  when we restrict the domain category to $\widetilde{\mathcal{O}}_c(\R^{m+1})^\op$. The result follows from \cite[Theorem 1.2]{FTW2020}.
\end{proof}

\subsection{Hairy Graph Complexes}
We start by briefly introducing the hairy graph complex $\text{HGC}_{A, d}$ associated to a $\Omega\mathbf{L}^{c}_{\infty}$ algebra $A$ and an integer $d\geq2$. For detailed introduction consult \cite[Section 3]{FTW2020}, \cite{Willwacher2023}.

A hairy graph $G = (V^E, f, V^I, E)$ decorated by $A$ is a connected graph where:
\begin{itemize}
  \item $V^E$ is a finite non-empty set of vertices, which are called external. The external vertices are required to have valence $1$;
  \item $f: V^E \to \bar{A}$ is a map which assigns to each external vertex $v\in V^E$ a homogeneous element of the augmentation ideal $\bar{A}$;
  \item $V^I$ is a finite set, disjoint from $V^E$, of vertices, which are called internal. The internal vertices are required to have valence at least 3;
  \item $E$ is a finite set of oriented edges. We allow tadpoles and multiple edges.
\end{itemize}
We assign degree $-d$ to the elements of $V^I$, degree $d-1$ to the elements of $E$ and degree $-\deg f(x)$ to every $x\in V^E$. The degree of a hairy graph is then given by
\[
\deg G = (d-1)|E| - d|V^I| - \sum_{x\in V^E} \deg f(x).
\]
Two oriented hairy graphs $G_1$ and $G_2$ are equivalent if there is a graph isomorphism $G_1\simeq G_2$ preserving decoration and orientation. Let $\mathcal{G}$ be the set of all equivalence classes of oriented hairy graphs. The space $\text{HGC}_{A,d}$ is the graded $\field$ vector space spanned by infinite linear combinations of the elements of $\mathcal{G}$, modulo the multilinearity relations with respect to the decorations of the external vertices and the orientation relations. 

We equip $\text{HGC}_{A,d}$ with a differential $\partial = \partial_A + \partial_\text{split} + \partial_\text{join}$, where $\partial_A$ is given by the differential of $A$ applied to the decoration, $\partial_\text{split}$ splits the internal vertices in all possible ways, $\partial_\text{join}$ merges together at least two external vertices replacing them by an internal vertex, a new external vertex decorated by the product of the merged decorations and an edge connecting those two vertices. Further, the vector space $\text{HGC}_{A,d}$ is equipped with the structure of $L_\infty$-algebra. This structure is functorial in $A$ and preserves quasi-isomorphisms. 
\begin{remark}\label{rmk:conventions}
  In the conventions that we use the $\mathbf{L}_{\infty}$ algebra $\text{HGC}_{A, d}$ is homologically graded. Further, it is shifted by $1$, i.e. all operations are of degree $-1$.
\end{remark}

\begin{lemma}\label{degree of graph}
  Let $G$ be a hairy graph decorated by a cdga $A$ concentrated in degrees $\leq m$. Then
  \[
  \deg G \geq \frac{d-3}{2}|V^I| + \frac{d-2m-1}{2}|V^E|.
  \]
\end{lemma}
\begin{proof}
  From the definition of a hairy graph it follows that
  \[
  2|E| \geq 3|V^I| + |V^E|.
  \]
  We combine it with the assumption that $A$ is concentrated in degrees $\leq m$ to get the result:
  \[
  \begin{split}
    \deg G = (d-1)|E| - d|V^I| - \sum_{x\in V^E} \deg f(x)
    &\geq \frac{d-1}{2}(3|V^I| + |V^E|) - d|V^I| - m|V^E| \\
    &= \frac{d-3}{2}|V^I| + \frac{d-2m-1}{2}|V^E|. \qedhere
  \end{split}
  \]
\end{proof}

\begin{lemma}\label{Lower bound on degree of Hairy Graphs}
  Let $d \geq 4$. Let $G$ be a hairy graph decorated by a cdga $A$ concentrated in degrees $\leq m$. Then
  \[
  \deg G \geq (d-m-2)|V^E|-d+3.
  \]
\end{lemma}
\begin{proof}
  We forget about the orientation and analyze what is the smallest possible degree of a graph $G$ with $r \geq 2$ external vertices.

  Notice that if an internal vertex of $G$ is of valence $\geq 4$ we can replace it by two internal vertices and an edge between them so that each new internal vertex is of valence $\geq 3$. This procedure reduces the degree of the graph by $1$. To bound the degree of $G$ from below we can assume that all internal vertices of $G$ are of valence exactly $3$, hence,
  \[
  2|E| = 3|V^I| + |V^E|.
  \]
  On the other hand, since $G$ is connected,
  \[
  |E| \geq |V^I| + |V^E| - 1.
  \]
  Combining the two expressions we get $3|V^I| + |V^E| \geq 2|V^I| + 2|V^E| - 2$, and so  $|V^I| \geq |V^E|-2$. Plugging it in \cref{degree of graph} we get the result.

  The bound can be improved when $|V^E| = 1$. Indeed, since the unique external vertex is required to be of valence $1$, there should be at least one internal vertex. It follows from \cref{degree of graph} that
  \[
  \deg G \geq \frac{d-3}{2}|V^I| + \frac{d-2m-1}{2}|V^E| \geq \frac{d-3}{2} + \frac{d-2m-1}{2} = d-m-2.
  \]
\end{proof}

In what follows we deal with algebras whose augmentation ideal $\bar{A}$ decomposes as a finite direct sum of cdga
\[
\bar{A} \simeq \bigoplus_{i\in\underline{r}} A_i.
\]
In such cases we will be interested in special kinds of graphs.

\begin{definition}
  Let $A$ be an augmented cdga equipped with the decomposition of its augmentation ideal $\bar{A} \simeq \bigoplus_{i\in\underline{r}} A_i$. We say that a hairy graph $G$ decorated by $A$ is \emph{homogeneous} if there is a map $\phi : V^E_G \to \underline{r}$ such that for every external vertex $v\in V^E_G$ the decoration $f(v)$ of $v$ belongs to $A_{\phi(v)}$.
  We say that a graph $G$ is \emph{surjective} homogeneous if $\phi$ is surjective.  \\ We write $\textnormal{HGC}_{A,d}^{\textnormal{surj}}$ to denote the $\mathbf{L}_\infty$ subalgebra of the $\mathbf{L}_\infty$ algebra $\textnormal{HGC}_{A,d}$ generated by surjective homogeneous graphs.
\end{definition}

\begin{remark}
  The graded vector spaces $\textnormal{HGC}_{A,d}$ (resp. $\textnormal{HGC}_{A,d}^{\textnormal{surj}}$) are spanned by homogeneous (resp. surjective homogeneous) graphs, as homogeneous graphs are stable under $\mathbf{L}_\infty$ operations.
\end{remark}

\begin{lemma}\label{Only finite number of graphs of fixed degree}
  Let $d \geq m+3 \geq 4$. Let $A$ be a finite dimensional augmented cdga concentrated in degrees $\leq m$, and equipped with a decomposition of its augmentation ideal $\bar{A} \simeq \bigoplus_{i\in\underline{r}} A_i$. Then $(\textnormal{HGC}_{A,d})_n$ is finite-dimensional for every $n\in\mathbb{Z}$. In particular, the same is true for $(\textnormal{HGC}_{A,d}^{\textnormal{surj}})_n$.
\end{lemma}

\begin{proof}
  Fix $n \in \mathbb{Z}$. For a hairy graph $G$ of degree $n$ we have that
  \[
  n \geq (d-m-2)|V^E|-d+3
  \]
  by \cref{Lower bound on degree of Hairy Graphs}. Since the right-hand side of the inequality is increasing in the number of external vertices, we see that there is only a finite number of possible values of $|V^E|$. It follows that the number of internal vertices is also bounded from above, since
  \[
  |V^I| \leq \frac{2}{d-3} \bigl(n - \frac{d-2m-1}{2}|V^E| \bigr)
  \]
  by \cref{degree of graph}. Further, from the definition of degree
  \[
  n = (d-1)|E| - d|V^I| - \sum_{x\in V^E} \deg f(x) \geq (d-1)|E| - d|V^I| - m|V^E|,
  \]
  showing that the number of edges is also bounded. We conclude that $(\textnormal{HGC}_{A,d})_n$ is finitely generated since $A$ is finite dimensional.
\end{proof}

\subsection{Rational models}

\begin{proposition}[{{\cite[Theorem 1.1]{FTW2020}}}]\label{rational homotopy type}
  Let $d \geq m+3$. There is a weak equivalence
  \[
  \textnormal{IBimod}^h_{\mathcal{F}_{m+1}}(\mathcal{IF}_U, \mathcal{IF}_d^\field) \simeq \textnormal{MC}_\bullet (\textnormal{HGC}_{\Omega(U_*),d}).
  \]
  It is natural with respect to $U \in \widetilde{\mathcal{O}}_c(\R^{m+1})$.
\end{proposition}

\begin{proof}
  In \cite[Theorem 1.1]{FTW2020} it was shown that for any $U \in \widetilde{\mathcal{O}}_c(\R^{m+1})$ there is the required weak equivalence. We analyse their proof to show that this weak equivalence is natural. At every step of the proof the authors do one of the following:
  \begin{itemize}
    \item replace the input, which does not depend on $U$, of a derived mapping space by a weakly equivalent object;
    \item use a Quillen adjunction;
    \item use \cite[Proposition 2.5]{FTW2020} to replace $\mathcal{IF}_U$ by $U_*^{\times\bullet}$.
  \end{itemize}
  All of these are clearly natural with respect to $U \in \widetilde{\mathcal{O}}_c(\R^{m+1})$.
\end{proof}

For an $\mathbf{L}_\infty$ algebra $\mathfrak{g}$ let $\mathfrak{g}_{>0}$ denote its positive degree truncation, i.e. $(\mathfrak{g}_{>0})_{n} = 0$ for $n \leq 0$, $(\mathfrak{g}_{>0})_{1} = \ker(\mathfrak{g}_{1} \xrightarrow{d} \mathfrak{g}_{0})$ and $(\mathfrak{g}_{>0})_{n} = \mathfrak{g}_{n}$ for $n \geq 2$.
\begin{corollary}[{{\cite[Corollary 1.3]{FTW2020}, \cite[Theorem 1.1, Corollary 1.3]{Berglund2015}}}]\label{rational homotopy groups}
  Let $d \geq m+3$. There is, for every $U \in \widetilde{\mathcal{O}}_c(\R^{m+1})$, an isomorphism (natural in $U$):
  \[
  \pi_*^\field \overline{\textnormal{Emb}}_c(U, \R^d)_\iota \simeq H_* \bigl((\textnormal{HGC}_{\Omega(U_*),d})_{>0}\bigr),
  \]
  where $\pi_i^\field$ is the rationalisation of the $i$-th homotopy group for $i\geq 2$ and $\pi_1^\field$ is the Malcev completion of the fundamental group equipped with the Baker--Campbell--Hausdorff product. Furthermore, there is an isomorphism
  \[
  H^{*}(\overline{\textnormal{Emb}}_c(U, \R^d)_\iota; \field) \simeq H^{*}_{\textnormal{CE}} \bigl((\textnormal{HGC}_{\Omega(U_*),d})_{>0}\bigr),
  \]
  whcih is also natural in $U$.
\end{corollary}

\section{String links}
For this section we fix the parameters $m \geq 1$ and $d \geq m+3$. We show that the spaces $\embci{\Rm \times -}$ assemble together into a homotopy $\FIs$ space $\SL$ of string links. We use the results of the previous sections to show that the rational homotopy/homology groups of the constructed space satisfy representation stability. We work over $\mathbb{Q}$.

\subsection{Structure maps}\label{Structure maps}
For a finite set $S \subset \Np$ we identify $\Rm \times S$ with a subspace of $\R^{m+1}$ via the map sending $(x, i)$ to $(x, \frac{1}{i})$. We write $\iota$ to denote composition of the subspace inclusion $\Rm \times S \subset \R^{m+1}$ and the standard linear inclusion $\R^{m+1} \subset \R^{d}$. Following the notation of \cref{subsection:setting} we introduce the following definition.
\begin{definition}\label{def:sl}
  For a finite set $S \subset \Np$ we define
  \[
  \SL(S) \coloneq \embci{\Rm \times S}
  \]
  to be the connected component of the space of smooth embeddings of $\Rm \times S$ into $\R^{d}$ which contains $\iota$. This space is pointed at $\iota$.
\end{definition}

\begin{definition}\label{cns:cube_sl}
  For finite sets $S, T \subset \Np$ such that $S \subset T$ we define a map
  \[
  \phi^{T}_{S}: \SL(T) \to \SL(S)
  \]
  given by precomposing with the subset inclusion $\Rm \times S \subset \Rm \times T$. With respect to these maps the spaces $\SL(-)$ form an $\cube^{\op}$ space which we denote $\SL$.
\end{definition}
\begin{proposition}\label{cns:fiop_sl}
  There exist structure maps
  \[
  \sigma^{*}: \SL(T) \to \SL(S) \text{ for a bijection } \sigma: S \to T
  \]
  such that $\cube^{\op}$ space $\SL$ can be promoted to a homotopy $\FI^{\op}$ space $\SL: \FI^{\op} \to \textnormal{Ho}(\mathbf{Top_*})$.
\end{proposition}
\begin{proof}
  Let
  \[
  \sigma: S = (s_1 < \dots < s_r) \to T = (t_1 < \dots < t_r)
  \]
  be a bijection between finite subsets of $\Np$. For $f = (f_{t_1}, \dots, f_{t_r}) \in \embci{\Rm \times T}$ one would like to define $\sigma^{*} (f) = (f_{\sigma(s_1)}, \dots, f_{\sigma(s_r)})$. The problem with this approach is that such $\sigma^{*} (f)$ does not satisfy the boundary conditions since at infinity $f_{\sigma(s)}$ coincides with $\Rm\times\{\frac{1}{\sigma(s)}\}$ and not with $\Rm\times\{\frac{1}{s}\}$. To fix this problem we will transform the target space $\R^d$. Let $\gamma: [0,1] \to \textnormal{Conf}(r, \R^{d-m})$ be a path such that $\gamma(0) = (\frac{1}{\sigma(s_1)}, \dots, \frac{1}{\sigma(s_r)})$ and $\gamma(1) = (\frac{1}{s_1}, \dots, \frac{1}{s_r})$. For any finite subset of $\R^{d-m}$ the restriction map
  \[
  \textnormal{ev}: \textnormal{Diff}(\R^{d-m}) \to \textnormal{Conf}(r, \R^{d-m})
  \]
  is a fibration by \cite{Pal60,Cer61}. Therefore, we can lift $\gamma$ to $\widetilde{\gamma}: [0,1] \to \textnormal{Diff}(\R^{d-m})$ such that $\widetilde{\gamma}(0) = \operatorname{id}_{\R^{d-m}}$ and $\textnormal{ev}(\widetilde{\gamma}) = \gamma$. In particular, $\widetilde{\gamma}(1)$ is a diffeomorphism of $\R^{d-m}$ sending $\frac{1}{\sigma(s_i)}$ to $\frac{1}{s_i}$ for $1 \leq i \leq r$. We set
  \[
  \sigma^{*} (f) = (\operatorname{id}_{\Rm} \times \widetilde{\gamma}(1)) \circ (f_{\sigma(s_1)}, \dots, f_{\sigma(s_r)}).
  \]

  We claim that this is well-defined. First, we check that this construction does not depend on the choice of the path $\gamma$. Let $\gamma': [0,1] \to \textnormal{Conf}(r, \R^{d-m})$ be another path satisfying the same boundary conditions, let $\widetilde{\gamma}'$ be its lift, and let $(\sigma')^{*}$ be defined as $(\sigma')^{*} = (\operatorname{id}_{\Rm} \times \widetilde{\gamma}'(1)) \circ (f_{\sigma(s_1)}, \dots, f_{\sigma(s_r)})$. Since $d-m \geq 3$ the space $\textnormal{Conf}(r, \R^{d-m})$ is simply-connected, and we can find a homotopy $H: [0,1]^2 \to \textnormal{Conf}(r, \R^{d-m})$ such that $H(0,t) = \gamma(t)$, $H(1,t) = \gamma'(t)$, $H(x,0) = (\frac{1}{\sigma(s_1)}, \dots, \frac{1}{\sigma(s_r)})$ and $H(x,1) = (\frac{1}{s_1}, \dots, \frac{1}{s_r})$. The lifting problem
  \[
  \begin{tikzcd}
    \left[0,1\right]\times\{0\} \cup \{0,1\}\times\left[0,1\right] \arrow[r] \arrow[d, hook]
    & \textnormal{Diff}(\R^{d-m}) \arrow[d, "\textnormal{ev}"] \\
    \left[0,1\right]^2 \arrow[r, "H"] \arrow[ur, dotted, "\widetilde{H}"]
    & \textnormal{Conf}(r, \R^{d-m})
  \end{tikzcd}
  \]
  where the top map is $\operatorname{id}\cup(\widetilde{\gamma},\widetilde{\gamma}')$ admits a solution. By construction $\widetilde{H}(-,1): [0,1] \to \textnormal{Diff}(\R^{d-m})$ is a path from $\widetilde{\gamma}(1)$ to $\widetilde{\gamma}'(1)$ fixing $(\frac{1}{s_1}, \dots, \frac{1}{s_r})$. This induces a pointed homotopy between $\sigma^{*}$ and $(\sigma')^{*}$, hence, they define the same map in $\textnormal{Ho}(\mathbf{Top_*})$.

  Clearly, $\operatorname{id}^{*} = \operatorname{id}$. For two composable bijections $S \xrightarrow{\sigma} T \xrightarrow{\tau} K$ the composition of the induced maps is
  \begin{align*}
    \sigma^{*}\tau^{*} (f) &= (\operatorname{id}_{\Rm} \times \widetilde{\gamma}_\sigma(1)) \circ (\operatorname{id}_{\Rm} \times \widetilde{\gamma}_\tau(1))
    \circ (f_{\tau\sigma(s_1)}, \dots, f_{\tau\sigma(s_r)}) \\
    &=
    (\operatorname{id}_{\Rm} \times (\widetilde{\gamma}_\sigma \circ \widetilde{\gamma}_\tau)(1)) \circ (f_{\tau\sigma(s_1)}, \dots, f_{\tau\sigma(s_r)}).
  \end{align*}
  Evaluating $\widetilde{\gamma}_\sigma \circ \widetilde{\gamma}_\tau$ on $(\frac{1}{\tau\sigma(s_1)}, \dots, \frac{1}{\tau\sigma(s_r)})$ we get a path $\gamma_{\tau\sigma}$ which can be used to define $(\tau\sigma)^{*}$. Since we have already shown that the lift is unique up to a pointed homotopy we get that $(\tau\sigma)^{*} = \sigma^{*}\tau^{*}$ in $\textnormal{Ho}(\mathbf{Top_*})$. For $S' \subset S$ let $\sigma': S' \to T'$ be the bijection given by the restriction of $\sigma: S \to T$. By construction, $\phi_{S'}^{S} \circ \sigma^{*} = (\sigma')^{*} \circ \phi_{T'}^{T}$ since we can use $\gamma_\sigma$ to define $(\sigma')^{*}$. Therefore, $\SL$ can be promoted to a functor $\FI^\op \to \textnormal{Ho}(\mathbf{Top_*})$.
\end{proof}
\begin{proposition}\label{cns:fis_sl}
  There exist structure maps
  \[
  \psi^{T}_{S}: \SL(S) \to \SL(T) \text{ for } S \subset T 
  \]
  such that homotopy $\FI^{\op}$ space $\SL$ can be promoted to a homotopy $\FIs$ space $\SL: \FIs \to \textnormal{Ho}(\mathbf{Top_*})$.
\end{proposition}
\begin{proof}
  We start by defining the maps $\psi_r: \embci{\Rm \times \underline{r}} \to \embci{\Rm \times \underline{r+1}}$ adding a copy of $\Rm$.
  Let
  \[
  \Psi_r: \R^{d-m} \xrightarrow{\sim} \{(x_1, \dots, x_{d-m}) \in \R^{d-m} \mid x_1 > \frac{2}{2r+1}\}
  \]
  be an diffeomorphism isotopic to the identity $\operatorname{id}_{\R^{d-m}}$ via a relative isotopy of embeddings fixing $(\frac{1}{i},0,\dots,0)$ for all $1 \leq i \leq r$. We define
  \[
  \psi_r (f) = ((\operatorname{id}_{\Rm} \times \Psi_r) \circ f, \iota),
  \]
  where $\iota(x) = (x, \frac{1}{r+1}, 0, \dots, 0)$. For $\sigma = (r+1 \; r+2) \in \Sigma_{r+2}$ we can choose a lift $\widetilde{\gamma}_\sigma$ whose support lies inside $\{(x_1, \dots, x_{d-m}) \in \R^{d-m} \mid x_1 \leq \frac{2}{2r+1}\}$. This shows that $\sigma^{*} \psi_{r+1} \psi_r \simeq \psi_{r+1} \psi_r$. A relative isotopy between $\Psi_r$ and $\operatorname{id}_{\R^{d-m}}$ induces a pointed homotopy between an endomorphism $\phi^{\underline{r+1}}_{\underline{r}} \psi_r$ of $\embci{\Rm \times \underline{r}}$ and the identity, hence, $\phi^{\underline{r+1}}_{\underline{r}} \psi_r = \operatorname{id}$. It is also clear that
  \[
  \phi^{\underline{r+1}}_{\underline{r+1}\setminus\{r\}} \psi_{r} \simeq \sigma^{*} \psi_{r-1} \phi^{\underline{r}}_{\underline{r-1}}
  \]
  where $\sigma: \underline{r+1}\setminus\{r\} \xrightarrow{\sim} \underline{r}$ is the obvious bijection sending $r+1$ to $r$.

  For $S \subset T$ there exists a bijection $\tau: T \xrightarrow{\sim} \underline{\lvert T \rvert}$ such that it restricts to a bijection $\sigma: S \xrightarrow{\sim} \underline{\lvert S \rvert}$. We set
  \[
  \psi^{T}_{S} \coloneq \tau^{*} \psi_{\lvert T \rvert-1} \dots \psi_{\lvert S \rvert} (\sigma^{*})^{-1}.
  \]
  Any map $f: S \to T$ in $\FIs$ can be uniquely decomposed as
  \[
  \begin{tikzcd}
    S \arrow[r, "f"] \arrow[d, dashed] & T \\
    \operatorname{dom} f \arrow[r, "\sim"] & \im f \arrow[u, hook]
  \end{tikzcd}
  \]
  where $\operatorname{dom} f$ is the domain of $f$, i.e. the set of elements on which $f$ is defined. We set
  \[
  f_{*} \coloneq \psi^{T}_{\im f} \circ (\sigma^{*})^{-1} \circ \phi^{S}_{\operatorname{dom} f}
  \]
  where $\sigma$ is the induced bijection between $\operatorname{dom} f$ and $\im f$. It is straightforward to check that this is well-defined.
\end{proof}

Another important property of the spaces of string links is that they are loop spaces.
\begin{lemma}\label{cns:loop_space}
  The space of string links $\SL(S)$ is a loop space for any finite $S \subset \Np$. Furthermore, the maps $\phi^{T}_{S}: \SL(T) \to \SL(S)$ are loop space maps.
\end{lemma}
\begin{proof}
  We use the recognition principle of \cite{BV1968}, \cite{May1972}.

  Let $I$ denote the open interval $(0,2)$. We write $\textnormal{Emb}_c(I^m \times \underline{r}, I^d)$ for the space of smooth embeddings of $r$ copies of an $I^m$ into a $I^d$ coinciding with a fixed embedding near boundary. We further assume that it sends interior of $I^m \times \underline{r}$ into interior of $I^d$. We write $\textnormal{Emb}_c(I^m \times \underline{r}, I^d)_\iota$ to denote the connected component containing the fixed standard embedding. The obvious inclusion $\textnormal{Emb}_c(I^m \times \underline{r}, I^d) \hookrightarrow \embc{\Rm \times \underline{r}}$ is a weak equivalence. In particular, its restriction $\textnormal{Emb}_c(I^m \times \underline{r}, I^d)_{\iota} \hookrightarrow \SL(\underline{r})$ is also a weak equivalence.

  Let $\C_m$ denote the little $m$-cubes operad. For an element $c: I^m \to I^m$ of $\C_m(1)$ we write $\tilde{c}: I^d \to I^d$ to denote $c \times \operatorname{id}_{I^{d-m}}$. For an element $c = (c_1,\dots,c_k)$ of $\C_m(k)$, an index $1 \leq j \leq r$ and maps $f_1,\dots,f_k: I^m \to I^d$ we define $c(f_1,\dots,f_k): I^m \to I^d$ by
  \[
  c^j(f_1,\dots,f_k)(x) =
  \begin{cases}
    f_i(c_i^{-1}(x)), & \text{if } x \in \operatorname{Im}(c_i) \text{ for some } 1 \leq i \leq k; \\
    \iota_j(x), & \text{otherwise}.
  \end{cases}
  \]
  Here $\iota_j$ denotes the restriction of $\iota$ to the $j$-th copy of $I^m$.
  The $\C_m$-algebra structure on $\textnormal{Emb}_c(I^m \times \underline{r}, I^d)$ is given by the maps
  \begin{align*}
    \C_m(k) \times \textnormal{Emb}_c(I^m \times \underline{r}, I^d)^{\times k} &\to \textnormal{Emb}_c(I^m \times \underline{r}, I^d) \\
    ((c_1,\dots,c_k),(f_1^1,\dots,f_r^1),\dots,(f_1^k,\dots,f_r^k)) &\mapsto (c^1(\tilde{c}_1f_1^1,\dots,\tilde{c}_kf_1^k),\dots,c^r(\tilde{c}_1f_r^1,\dots,\tilde{c}_kf_r^k))
  \end{align*}
  This structure restricts to the connected component $\textnormal{Emb}_c(I^m \times \underline{r}, I^d)_{\iota}$, hence, the latter is an $m$-fold loop space. Furthermore, the restriction maps $\phi_r: \textnormal{Emb}_c(I^m \times \underline{r+1}, I^d)_{\iota} \to \textnormal{Emb}_c(I^m \times \underline{r}, I^d)_{\iota}$ are $\C_m$-algebra maps.
\end{proof}

Similarly to \cref{def:sl} we define the objects $\iSL(-)$ and $\oSL(-)$.
\begin{definition}
  Let $S$ be a finite subset of $\Np$. We define
  \[
  \iSL(S) \coloneq \immci{\Rm \times S}
  \]
  to be the connected component of the space of smooth immersions of $\Rm \times S$ into $\R^{d}$ which contains $\iota$. It is pointed at $\iota$. We also define
  \[
  \oSL(S) \coloneq \oembci{\Rm \times S}
  \]
  to be the connected component of the base point of the homotopy fiber of the tautological inclusion $\embc{\Rm \times S} \to \immc{\Rm \times S}$.
\end{definition}
\begin{remark}
  Remember that we have fixed a model for the homotopy fiber of a map $f: (X,x_{0}) \to (Y,y_{0})$ by setting
  \[
  \operatorname{hofib} f = \{(x, \gamma) \in X \times Y^{[0,1]} \mid x \in X; \, \gamma:[0,1] \to Y, \gamma(0)=f(x), \gamma(1)=y_{0}\}.
  \]
  We also assume that it is equipped with a base point $(x_{0}, c_{y_{0}})$, where $c_{y_{0}}$ denotes the constant path at $y_{0}$.
\end{remark}
\begin{remark}\label{rmk:cns_sl}
  The maps outlined in \cref{cns:cube_sl,cns:fiop_sl,cns:fis_sl} can be analogously defined for $\iSL(-)$ and $\oSL(-)$. This defines homotopy $\FIs$ spaces $\iSL$ and $\oSL$. Moreover, \cref{cns:loop_space} can be adapted to show that the spaces $\iSL(S)$ and $\oSL(S)$ are also loop spaces and the maps $\phi^{T}_{S}$ are also the loop space maps.
\end{remark}
For any finite set $S \subset \Np$ there is a sequence
\begin{equation}\label{eq:seq_sl}
  \oSL(S) \to \SL(S) \to \iSL(S)
\end{equation}
where the first map is the restriction of the projection and the second map is the tautological inclusion. It is clear that this sequence is natural with respect to the $\FIs$ structure.
\begin{proposition}\label{prop:rational_les}
  The sequence of $\FIs$-modules
  \[
  \pi^\field_{n}\oSL \to \pi^\field_{n}\SL \to \pi^\field_{n}\iSL
  \]
  associated to \eqref{eq:seq_sl} is exact for every $n\geq1$.
\end{proposition}
\begin{proof}
  By definition the sequence
  \[
  \oembc{\Rm \times S} \to \embc{\Rm \times S} \to \immc{\Rm \times S}
  \]
  is a homotopy fiber sequence for every finite $S \subset \Np$. The associated long exact sequence of homotopy groups is
  \begin{multline*}
    \dots \to \pi_{n+1}\iSL(S) \to \pi_{n}\oSL(S) \to \pi_{n}\SL(S) \to \pi_{n}\iSL(S) \to \dots \\
    \dots \to \pi_{1}\oSL(S) \to \pi_{1}\SL(S) \to \pi_{1}\iSL(S)
  \end{multline*}
  since the spaces $\oSL(S), \SL(S)$ and $\iSL(S)$ are the corresponding components of the spaces $\oembc{\Rm \times S}, \embc{\Rm \times S}$ and $\immc{\Rm \times S}$ respectively. It is natural with respect to the $\FIs$ structure maps. By \cref{rmk:cns_sl} all the spaces involved are loop spaces, hence, all the groups are abelian. In particular, the rationalisation
  \begin{multline*}
    \dots \to \pi^{\field}_{n+1}\iSL \to \pi^{\field}_{n}\oSL \to \pi^{\field}_{n}\SL \to \pi^{\field}_{n}\iSL \to \dots \\
    \dots \to \pi^{\field}_{1}\oSL \to \pi^{\field}_{1}\SL \to \pi^{\field}_{1}\iSL
  \end{multline*}
  is also exact.
\end{proof}

\subsection{Reformulation in terms of graph complexes}
\begin{definition}
  We write $\widetilde{H}^*((\mathbb{S}^m)^{\vee \bullet})$ to denote the dg $\FIs$-algebra of \cref{lem:ind_of_algebra} associated to cdga $\widetilde{H}^*(\mathbb{S}^m)$. We write $\slgraph$ to denote the $\FIs$ $\mathbf{L}_\infty$ algebra $(\textnormal{HGC}_{H((\mathbb{S}^m)^{\vee {\bullet}}),d})_{>0}$.
\end{definition}
\begin{remark}
  Note that the product in $\widetilde{H}^*(\mathbb{S}^m)$ is trivial. It implies that in $\slgraph(S)$ all the brackets vanish and, hence, it is just a differential graded vector space. 
\end{remark}
\begin{proposition}\label{Symmetrization statement}
  The graded $\cube^\op$-modules $\Res H_*(\slgraph)$ and $\Res \pi_*^\field \oSL$ are isomorphic.
\end{proposition}
\begin{proof}
  Let $S \subset \Np$ be a finite set and let
  \[
  M_S = \Biggl(\Rm \times \bigcup_{i \in S} \biggl(\frac{2i+1}{2i(i+1)}, \frac{2i+3}{2i(i+1)}\biggr)\Biggr) \cup \bigl(\R^{m+1} \setminus \overline{B}(0, 2)\bigr),
  \]
  where $\overline{B}(0, 2) \subset \R^{m+1}$ is the closed ball with center $0$ and radius $2$. For $S \subset T$ there is a map $\oembci{M_T} \to \oembci{M_S}$ induced by the restriction to $M_S \subset M_T$. Therefore, the spaces $\oembci{M_-}$ assemble together into a functor from $\cube^\op$ to $\mathbf{Top}_*$.

  Recall that we identify $\Rm \times S$ with a subspace of $\R^{m+1}$ so that $\Rm \times S$ is a subspace of $M_S$. By abuse of notation let $\iota_S: \Rm \times S \hookrightarrow M_S$ denote the subspace inclusion. The restriction map
  \[
  \iota_S^*: \oembci{M_S} \to \oembci{\Rm \times S}
  \]
  is a weak equivalence of pointed topological spaces by \cite[Proposition 2.2]{STT2018}. This map is natural in $S$, so it induces a natural weak equivalence of $\cube^\op$ spaces $\oembci{M_-}$ and $\Res \oSL$.

  By \cref{rational homotopy groups} there is a natural in $S$ isomorphism
  \[
  \pi_*^\field \overline{\textnormal{Emb}}_c(M_S, \R^d)_\iota \simeq H_* \bigl((\textnormal{HGC}_{\Omega(\hat{M}_S),d})_{>0}\bigr),
  \]
  where $\hat{M}_S = M_S \cup \{\infty\}$ is a subspace of the one point compactification $\R^{m+1} \cup \{\infty\}$ of $\R^{m+1}$. To finish the proof we show that the $\mathbf{L}_\infty$ $\cube^\op$-algebras $\textnormal{HGC}_{H((\mathbb{S}^m)^{\vee {\bullet}}),d}$ and $\textnormal{HGC}_{\Omega(\hat{M}_S),d}$ are quasi isomorphic. Since the spheres $\mathbb{S}^{m}$ are formal there is an augmented cdga $A$ and a zigzag of weak equivalences
  \[
  \overline{\Omega}^{*}(\mathbb{S}^{m}) \leftarrow \overline{A} \to \widetilde{H}^{*}(\mathbb{S}^{m}).
  \]
  Using this we write the following sequence of weak equivalences
  \[
  \overline{\Omega}^{*}(\hat{M}_S) \leftarrow \overline{\Omega}^{*}((\mathbb{S}^m)^{\vee S}) \to \overline{\Omega}^{*}(\mathbb{S}^m)^{\times S} \leftarrow \bar{A}^{\times S} \to \widetilde{H}^{*}(\mathbb{S}^{m})^{\times S} \leftarrow \widetilde{H}^{*}((\mathbb{S}^m)^{\vee S}),
  \]
  where the first map is induced by the deformation retraction of $\hat{M}_S$ onto $(\Rm \times S) \cup \{\infty\} \simeq (\mathbb{S}^m)^{\vee S}$, the second map is induced by inclusions $\mathbb{S}^m \to (\mathbb{S}^m)^{\vee S}$. This sequence is clearly natural with respect to the subset inclusions $S \subset T$. Therefore, the dg $\cube^{\op}$-algebras $\widetilde{H}^{*}((\mathbb{S}^m)^{\vee {\bullet}})$ and $\overline{\Omega}^{*}(\hat{M}_{\bullet})$ are weakly equivalent. The proof is finished since the Hairy Graph Complex construction is natural and preserves weak equivalences.
\end{proof}

\subsection{Representation stability}\label{Section on Representation Stability}

The aim of this subsection is to show that rational homotopy and cohomology groups of the spaces $\SL$ satisfy representation stability. We start by investigating $\slgraph$.

\begin{definition}
  Let $\slgraph^{\textnormal{surj}}$ be the dg $\FB$-submodule of $\Res \slgraph$ such that $\slgraph^{\textnormal{surj}}(\underline{r})$ is spanned by surjective homogeneous graphs with respect to the decomposition $\widetilde{H}((\mathbb{S}^m)^{\vee {\underline{r}}}) \simeq \widetilde{H}(\mathbb{S}^m)^{\bigoplus {\underline{r}}}$. In other words, $\slgraph^{\textnormal{surj}}(\underline{r})$ is the underlying dg module of $(\textnormal{HGC}_{H((\mathbb{S}^m)^{\vee \underline{r}}),d}^\textnormal{surj})_{>0}$.
\end{definition}
\begin{proposition}\label{Computation for graph complexes of string links}
  The $\FIs$-module $H_n(\slgraph)$ is finitely generated and generated in arity $\leq \frac{n+d-3}{d-m-2}$.
\end{proposition}
\begin{proof}
  The dg $\FIs$-module $\slgraph$ is induced from the dg $\FB$-module $\crosseffect(\slgraph)$ by
  \cref{Equivalence of FB and FI_sharp modules,injective and projective homology}.
  A partial map $f:T \to T\setminus\{t\}$ acts on a homogeneous graph $G \in \slgraph(T)$ by projecting the decorations $\widetilde{H}(\mathbb{S}^m)^{\oplus T}$ onto $\widetilde{H}(\mathbb{S}^m)^{\oplus T\setminus\{t\}}$. Therefore, $\crosseffect(\slgraph)(T) = \slgraph^\textnormal{surj}(T)$ since $\slgraph(T)$ is spanned by homogeneous graphs. By \cref{homology of the free functors} the graded $\FIs$-module $H_*(\slgraph)$ is naturally isomorphic to $\Ind H_*(\slgraph^\textnormal{surj})$. To finish the proof it is enough to show that $(\slgraph^\textnormal{surj})_n$ is finitely generated in each arity and is concentrated in arities $\leq \frac{n+d-3}{d-m-2}$. The first part follows from Lemma \ref{Only finite number of graphs of fixed degree}. For the second part, notice that $\slgraph^\textnormal{surj}(\underline{r})$ is generated by graphs with at least $r$ external vertices and by \cref{Lower bound on degree of Hairy Graphs} the degree of such a graph is $\geq (d-m-2)r-d+3$.
\end{proof}
\begin{theorem}\label{thm:homotopy_sl}
  Let $d \geq m+3 \geq 4$. The $\FIs$-module $\pi_n^\field\SL$ is finitely generated and generated in arity $\leq \frac{n+d-3}{d-m-2}$ for any $n \geq 1$.
\end{theorem}
\begin{proof}
  Fix $n \geq 1$.
  By \cref{prop:rational_les} the sequence of $\FIs$-modules
  \[
  \pi_n^\field\oSL \to \pi_n^\field\SL \to \pi_n^\field\iSL
  \]
  is exact. In view of \cref{lem:thick} to prove the theorem it is enough to show that $\pi_n^\field\oSL$ and $\pi_n^\field\iSL$ satisfy it as well.

  The $\FIs$-module $H_n(\slgraph)$ satisfies the statement of the theorem (\cref{Computation for graph complexes of string links}), hence, so does $\pi_n^\field\oSL$ by \cref{lem:properties_cube,Computation for graph complexes of string links}. For the space of immersions we have that
  \[
  \iSL(S) \simeq \iSL(\underline{1})^{\times S}.
  \]
  In particular,
  \[
  \pi_n^\field\iSL(S) \simeq (\pi_n^\field\iSL(\underline{1}))^{\oplus S}.
  \]
  Therefore, the $\FIs$-module $\iSL$ is generated in arity $1$. To show that it is finitely generated note that it follows from the Smale--Hirsch theory \cite{Hir59} that there is a homotopy equivalence
  \[
  \immc{\Rm} \simeq \Omega^m V_m^d,
  \]
  where $V_m^d$ is the Stiefel manifold of $m$-frames in $\R^d$. By \cite{Ser53} the vector spaces $\pi_n^\field \Omega^m V_m^d$ are finite dimensional.
\end{proof}
\begin{theorem}\label{thm:cohomology_sl}
  Let $d \geq m+3 \geq 4$. The graded $\FIs$-module $H^*(\SL; \field)$ is of finite type with slope $\leq \frac{d-2}{d-m-2}$.
\end{theorem}
\begin{proof}
  The $\FIs$-module $\pi_n^\field\SL$ is finitely generated and generated in arity $\leq \frac{n+d-3}{d-m-2}$ for any $n \geq 1$ by \cref{thm:homotopy_sl}. In particular, the graded $\FIs$-module $\pi_{*}^\field\SL$ is of finite type with slope $\leq \frac{d-2}{d-m-2}$ since
  \[
  \frac{n+d-3}{d-m-2} \leq n\frac{d-2}{d-m-2},
  \]
  for $n \geq 1$. We deduce that so is its dual $(\pi_{*}^\field\SL)^{\vee}$ (\cref{prop:properties_dual}) and, hence, the free graded symmetric $\FIs$-algebra $\operatorname{Sym}(\pi_{*}^\field\SL)^{\vee}$ generated by it (\cref{prop:algebra_slope}). In view of \cref{cns:loop_space} the underlying graded $\cube$-algebra of the latter is isomorphic to $\Res H^*(\SL)$. The proof is finished by \cref{lem:properties_cube}.
\end{proof}
\begin{proposition}\label{prop:bound_improvement_sl}
  Assume that the hypothesis of \cref{thm:cohomology_sl} holds. If $d \geq 2m+2$ then the slope can be improved to $\leq \frac{2}{d-2m-1}$.
\end{proposition}
\begin{proof}
  We need to find $\kappa$ such that 
  \[
  \frac{n+d-3}{d-m-2} \leq n\kappa
  \]
  for all possible values of $n$. Since the spaces $\embc{\R^{m} \times S}$ are $(d-2m-2)$-connected (follows from \cite[Corollary 4.5, Theorem 4.8]{Koytcheff2025} and \cite[Proposition 3.9]{Budney2008}) it is enough to consider $n \geq d-2m-1$.
  The function 
  \begin{equation}\label{eq:slope}
    \frac{1}{n}\frac{n+d-3}{d-m-2}
  \end{equation}
  is decreasing with $n$, hence, we just consider the case $n = d-2m-1$. Plugging it into \eqref{eq:slope} we get
  \[
  \frac{1}{d-2m-1}\frac{2d-2m-4}{d-m-2} = \frac{2}{d-2m-1}. \qedhere
  \]
\end{proof}

\section{Links of Manifolds}
Let $d \geq 3$ and let $M$ be a closed smooth manifold of dimension together with an embedding $\iota: M \to \R^{d-2}$. In this section we show that the spaces $\embi{M \times S}$ assemble into a pointed homotopy $\FIs$ space $\ML$. Analogously to the previous section, we show that the rational homotopy/homology groups of the constructed spaces satisfy representation stability. We work over $\mathbb{Q}$.

\subsection{Structure maps}
Let $M$ be a closed manifold and let $\iota: M \to D^{d-2} \subset \R^{d-2}$ be a fixed embedding whose image is contained in the unit ball. For $s \in \Np$ we define a new embedding
\[
M \times \{s\} \to M \xrightarrow{\iota} D^{d-2} \to \R^{d-2},
\]
where the first map is the projection and the third map is the composition of scaling by $\frac{1}{2^{s+1}}$ and translating by $(0,\dots,0,\frac{1}{2^{s-1}})$. We identify $M \times \{s\}$ with its image under this embedding. Further, for a finite set $S \subset \Np$ we identify $M \times S$ with a subspace of $\R^{d-2}$ corresponding to $\coprod_{s \in S} M \times \{s\}$. By abuse of notation we write $\iota: M \times S \to \R^{d}$ to denote the composition of the subspace inclusion $M \times S \subset \R^{d-2}$ and the standard linear inclusion $\R^{d-2} \subset \R^{d}$.

Similarly to \cref{def:sl} we introduce the following definitions.
\begin{definition}
  For a finite set $S \subset \Np$ we define
  \[
  \ML(S) \coloneq \embi{M \times S}
  \]
  to be the connected component of the space of smooth embeddings of $M \times S$ into $\R^{d}$ which contains $\iota$. This space is pointed at $\iota$. Similarly, we define
  \[
  \iML(S) \coloneq \immi{M \times S}
  \]
  to be the corresponding component of the space of smooth immersions. It is also pointed at $\iota$. Finally, we define
  \[
  \oML(S) \coloneq \oembi{M \times S}
  \]
  to be the connected component of the base point of the homotopy fiber of the obvious inclusion $\ML(S) \to \iML(S)$.
\end{definition}
As in \cref{eq:seq_sl} for any finite set $S \subset \Np$ there is a sequence
\begin{equation}\label{eq:seq}
  \oML(S) \to \ML(S) \to \iML(S)
\end{equation}
where the first map is the restriction of the projection and the second map is the tautological inclusion.
\begin{definition}
  For finite sets $S, T \subset \Np$ such that $S \subset T$ we define maps
  \[
  \phi^{T}_{S}: \ML(T) \to \ML(S)
  \]
  given by precomposing with the subset inclusion $M \times S \subset M \times T$. By abuse of notation we also write $\phi^{T}_{S}$ for the corresponding maps $\oML(T) \to \oML(S)$ and $\iML(T) \to \iML(S)$. With respect to these maps the spaces $\ML(-)$ (resp., $\oML(-)$, $\iML(-)$) form an $\cube^{\op}$ space which we denote $\ML$ (resp., $\oML$, $\iML$).
\end{definition}
The maps of \eqref{eq:seq} are clearly natural. Hence, they assemble into a sequence of maps of $\cube^{\op}$ spaces $\oML \to \ML \to \iML$.
\begin{proposition}
  There exist structure maps
  \begin{itemize}
    \item $\sigma_{*}: \ML(S) \to \ML(S')$ for a bijection $\sigma: S \to S'$;
    \item $\psi^{T}_{S}: \ML(S) \to \ML(T)$ for $S \subset T$;
  \end{itemize}
  such that $\cube^{\op}$ space $\ML$ can be promoted to a homotopy $\FIs$ space, which by abuse of notation we also denote $\ML$. The same holds for $\oML$ and $\iML$. Furthermore, the maps in \eqref{eq:seq} are natural and produce a sequence
  \[
  \oML \to \ML \to \iML
  \]
  of maps of homotopy $\FIs$ spaces.
\end{proposition}
\begin{proof}
  The proofs of \cref{cns:fiop_sl,cns:fis_sl} adapt. For the construction of $\sigma_{*}$ one shall consider the evaluation map
  \[
  \operatorname{ev}: \operatorname{Diff}_{c}(\R^{d}) \to \emb{\overline{D}^{d-2} \times S},
  \]
  where $\overline{D}^{d-2} \times \{s\}$ is identified with the unit ball of $\R^{d-2} \subset \R^{d}$ scaled by $\frac{1}{2^{s+1}}$ and translated by $(0,\dots,0,\frac{1}{2^{s-1}})$. It is again a fibration by \cite{Pal60,Cer61}. Moreover, the space $\emb{\overline{D}^{d-2} \times S}$ is weakly equivalent to the space of $(d-2)$-framed configurations 
  \[
  \operatorname{Conf}^{(d-2)\text{-fr}}(S, \R^{d}),
  \]
  which is simply connected for $d\geq3$. For the construction of $\psi_{r}$ one chooses $\Psi_{r}$ that is identity on $\overline{D}^{d-2} \times \underline{r}$.
\end{proof}
\begin{remark}
  In case of manifold links there is no boundary condition stopping us from defining the strict action of symmetric group permuting copies of $M$. However, we want our action to be base point preserving, so we still opt to choose an action defined only up to homotopy.
\end{remark}

For $s \in \Np$ let $\tML(\{s\})$ denote the covering space of $\iML(\{s\})$ corresponding to the subgroup
\[
\im(\pi_{1}\ML(\{s\}) \to \pi_{1}\iML(\{s\}))
\]
of $\pi_{1}\iML(\{s\})$. It is pointed at a lift of $\iota$.
\begin{proposition}\label{cns:lift}
  There exists a homotopy $\FIs$ space $\tML$ such that for all finite $S \subset \Np$
  \[
  \tML(S) = \prod_{s \in S} \tML(\{s\}).
  \]
  There also exists a map of homotopy $\FIs$ spaces $\ML \to \tML$ such that the sequence
  \begin{equation}\label{eq:hmtp_fi_seq}
    \oML \to \ML \to \tML
  \end{equation}
  is a homotopy fiber sequence in each arity.
\end{proposition}
\begin{proof}
  We start by showing that
  \begin{equation}\label{eq:pi1_image}
    \im(\pi_{1}\ML(S) \to \pi_{1}\iML(S)) = \prod_{s \in S} \im(\pi_{1}\ML(\{s\}) \to \pi_{1}\iML(\{s\})),
  \end{equation}
  where we identify $\pi_{1}\iML(S)$ and $\prod_{s \in S} \pi_{1}\iML(\{s\})$. In the following commutative diagram
  \[
  \begin{tikzcd}
    \ML(S) \arrow[r] \arrow[d, "\prod \phi^{S}_{\{s\}}"] & \iML(S) \arrow[d, "\prod \phi^{S}_{\{s\}}"] \\
    \prod_{s \in S} \ML(\{s\}) \arrow[r] & \prod_{s \in S} \iML(\{s\})
  \end{tikzcd}
  \]
  the right vertical map is a homeomorphism. Hence, we need to show that the left vertical map is surjective on $\pi_{1}$. For every $s\in\Np$ we set $I_{s}$ to be an open interval such that
  \begin{itemize}
    \item the image $\iota(M \times \{s\})$ is contained in $\R^{d-3} \times I_{s} \times \R^{2}$,
    \item $I_{s} \cap I_{t} = \emptyset$ when $s \neq t$.
  \end{itemize}
  The map
  \[
  \operatorname{Emb}_{c}(M\times\{s\}, \R^{d-3} \times I_{s} \times \R^{2})_{\iota} \to \ML(\{s\})
  \]
  given by post-composition with the inclusion $\R^{d-3} \times I_{s} \times \R^{2} \subset \R^{d}$ is a weak equivalence. In particular, any element of $\pi_1 \ML(\{s\})$ can be represented by a loop $\gamma_s$ such that for any $t$ the image of $\gamma_{s}(t)$ is contained in $\R^{d-3} \times I_{s} \times \R^{2}$. Therefore, any element of $\prod_{s \in S} \pi_{1} \ML(\{s\})$ can be represented by a loop in $\ML(S)$ given by the combination of $\gamma_{s}$ for $s \in S$.
  We set
  \[
  \tML(S) \coloneq \prod_{s \in S} \tML(\{s\})
  \]
  which by \eqref{eq:pi1_image} is the covering space of $\iML(S)$ corresponding to the subgroup $\im(\pi_{1}\ML(S) \to \pi_{1}\iML(S))$ of $\pi_{1}\iML(S)$. The spaces $\tML(S)$ are also pointed as products of pointed spaces. By definition the subspace inclusion maps $\ML(S) \to \iML(S)$ uniquely lift to the maps $\ML(S) \to \tML(S)$. For a partial injection $f:S \to T$ there is a commutative diagram
  \[
  \begin{tikzcd}
    \ML(S) \arrow[r] \arrow[rr, bend left=15] \arrow[d, "f_{*}"] & \iML(S) \arrow[d, "f_{*}"] & \tML(S) \arrow[l] \arrow[d, dashed] \\
    \ML(T) \arrow[r] \arrow[rr, bend right=15] & \iML(T) & \tML(T) \arrow[l] 
  \end{tikzcd}
  \]
  where the top and the bottom maps are surjective on $\pi_{1}$. Hence, $f_{*}$ lifts uniquely. Furthermore, all the homotopies also lift producing a homotopy $\FIs$ space $\tML$. By the uniqueness of lifts it follows that the maps $\ML(S) \to \tML(S)$ assemble into a map of homotopy $\FIs$ spaces $\ML \to \tML$. Finally, it is clear that the homotopy fiber of the map $\ML(S) \to \tML(S)$ is $\oML(S)$. Therefore, there is a sequence of homotopy $\FIs$ spaces
  \[
  \oML \to \ML \to \tML
  \]
  which is a homotopy fiber sequence in each arity.
\end{proof}

\subsection{Graphs}
For a finite set $S \subset \Np$ we set $M_S$ to be a regular neighbourhood of $M \times S$ in $\R^{d-2}$ and $\widetilde{M}_S$ to be the union of $M_S$ with a fixed antiball $\R^{d-2} \setminus D^{d-2}$. Here we require the antiball to be big enough so that it does not intersect $M_S$ for any $S$. The spaces $\oembi{M_{-}}$ and $\oembi{\widetilde{M}_{-}}$ assemble together into a functor from $\cube^\op$ to $\mathbf{Top}_*$ with maps induced by inclusions $S \subset T$.
\begin{lemma}[{{\cite[Proposition 2.2]{STT2018}}}]\label{lem:regular_neigh}
  The restriction maps
  \[
  \oembci{\widetilde{M}_S} \to \oembi{M_S} \to \oembi{M \times S},
  \]
  induced by the subspace inclusions $M \times S \hookrightarrow M_S \hookrightarrow \widetilde{M}_S$, are weak equivalences. Therefore, the $\cube^{\op}$-spaces $\oembci{\widetilde{M}_-}$, $\oembi{M_-}$ and $\Res \oML$ are naturally weakly equivalent.
\end{lemma}
Let $\hat{M}_S = \widetilde{M}_S \cup \{\infty\}$ be a subspace of the one point compactification $\R^{d-2} \cup \{\infty\}$ of $\R^{d-2}$. The spaces $\hat{M}_-$ also assemble into $\cube$-space. In particular, there is a $\cube^{\op}$ $\mathbf{L}_\infty$ algebra $(\textnormal{HGC}_{\Omega(\hat{M}_S),d})_{>0}$.
\begin{definition}
  Let $\A$ denote a $\mathbf{C}_\infty$ algebra obtained by transferring the cdga structure from $\Omega^*(M)$ to $H^*(M)$. We write $\A^{\times \bullet}$ for the $\FIs$ $\mathbf{C}_\infty$ algebra defined in \cref{lem:ind_of_algebra}. We write $\mlgraph$ for the $\FIs$ $\mathbf{L}_\infty$ algebra $(\textnormal{HGC}_{\A^{\times \bullet},d})_{>0}$.
\end{definition}
\begin{remark}\label{rmk:strict_morphisms}
  Here we regard $\A$ as a non-unital $\mathbf{C}_\infty$ algebra and equip its finite products with componentwise operations. The structure maps are projections, permutations, and extensions by zero, hence are strict $\mathbf{C}_\infty$ morphisms. By \cite[Section 3.4]{FTW2020}, the induced structure maps of $\mlgraph$ are strict $\mathbf{L}_\infty$ morphisms.
\end{remark}
\begin{lemma}\label{lem:graph_replacement}
  There is a natural weak equivalence of $\cube^{\op}$ $\mathbf{L}_\infty$ algebras $(\textnormal{HGC}_{\Omega(\hat{M}_S),d})_{>0}$ and $\Res\mlgraph$.
\end{lemma}
\begin{proof}
  We start by observing that there is a sequence of weak equivalences
  \[
  \overline{\Omega}^*(\hat{M}_S) \simeq \Omega^*(M_S) \simeq \Omega^*(M)^{\times S}
  \]
  which are compatible with the subset inclusions $S \subset T$. In particular, there is a commutative diagram of $\mathbf{C}_\infty$ algebras
  \[
  \begin{tikzcd}
    \A^{\times T} \arrow[d] & * \arrow[l] \arrow[d] \arrow[r] & \Omega^*(M)^{\times T} \arrow[d] \arrow[r] & \overline{\Omega}^*(\hat{M}_T) \arrow[d] \\
    \A^{\times S} & * \arrow[l] \arrow[r] & \Omega^*(M)^{\times S} \arrow[r] & \overline{\Omega}^*(\hat{M}_S)
  \end{tikzcd}
  \]
  where the horizontal maps are weak equivalences, the right most vertical map is induced by the inclusion $\hat{M}_S \hookrightarrow \hat{M}_T$ and all the other vertical maps are projections induced by the subset inclusion $S \subset T$. Therefore, there is a natural weak equivalence of $\cube^{\op}$ $\mathbf{C}_{\infty}$ algebras $\overline{\Omega}^{*}(\hat{M}_S)$ and $\Res \A^{\times \bullet}$. To finish the proof note that the Hairy Graph Complex construction is natural and preserves weak equivalences.
\end{proof}
\begin{proposition}\label{prop:passing_graphs_ml}
  The graded $\cube^{\op}$-modules $\Res \pi_*^\field \oML$ and $\Res H_*(\mlgraph)$ are naturally isomorphic when we restrict our attention to $* \geq 2$. Similarly, the graded $\cube$-modules $\Res H^{*}(\oML)$ and $\Res H^{*}_{\textnormal{CE}}(\mlgraph)$ are naturally isomorphic.
\end{proposition}
\begin{proof}
  By \cref{rational homotopy groups} there is an isomorphism of graded $\cube^{\op}$-modules 
  \[
  \pi_*^\field \overline{\textnormal{Emb}}_c(\widetilde{M}_S, \R^d)_\iota \simeq H_* \bigl((\textnormal{HGC}_{\Omega(\hat{M}_S),d})_{>0}\bigr),
  \]
  and an isomorphism of graded $\cube$-modules 
  \[
  H^{*}(\overline{\textnormal{Emb}}_c(\widetilde{M}_S, \R^d)_\iota; \field) \simeq H^{*}_{\textnormal{CE}} \bigl((\textnormal{HGC}_{\Omega(\hat{M}_S),d})_{>0}\bigr).
  \]
  The proposition follows from \cref{lem:regular_neigh,lem:graph_replacement}.
\end{proof}

\subsection{Representation stability}
\begin{lemma}\label{lem:computation_graphs_ml}
  The $\FIs$-module $\mlgraph_n$ is finitely generated in arity $\leq \frac{n+d-3}{d-m-2}$. In particular, the same is true for its homology $H_{n}(\mlgraph)$.
\end{lemma}
\begin{proof}
  The proof of \cref{Computation for graph complexes of string links} adapts.
\end{proof}
\begin{theorem}\label{thm:homotopy_ml}
    Let $d \geq 3$ and let $\iota: M \hookrightarrow \R^{d-2}$ be an embedding of a closed smooth manifold of dimension $m$. The $\FIs$-module $\pi^{\field}_{n}(\ML)$ is finitely generated in arity $\leq \frac{n+d-3}{d-m-2}$ for any $n \geq 2$.
\end{theorem}
\begin{proof}
  Most of the proof of \cref{thm:homotopy_sl} adapts. By \cref{cns:lift} there is a sequence
  \[
  \oML \to \ML \to \tML
  \] 
  which is a homotopy fiber sequence. It induces an exact sequence of $\FIs$-modules
  \[
  \pi^{\field}_{n}\oML \to \pi^{\field}_{n}\ML \to \pi^{\field}_{n}\iML
  \]
  for $n \geq 2$. The $\FIs$-module $\pi^{\field}_{n}\oML$ satisfies the statement of the theorem by \cref{lem:computation_graphs_ml}. The $\FIs$-module $\pi^{\field}_{n}\iML$ is generated in arity $1$. It is left to show that $\pi^{\field}_{n}\iML$ is finitely generated which will follow from the fact that $\pi_{n}\imm{M}$ is a finitely generated abelian group. By Smale--Hirsch theory \cite{Hir59} there is a sequence of weak equivalences 
  \[
  \imm{M} \to \operatorname{Imm}^{f}(M, \R^{d}) \to \Gamma(\operatorname{Mon}(TM, T\R^{d}), M),
  \]
  where $\operatorname{Imm}^{f}(M, \R^{d})$ is the space of formal immersions and $\Gamma(\operatorname{Mon}(TM, T\R^{d}), M)$ is the space of sections of the fiber bundle $\operatorname{Mon}(TM, T\R^{d}) \to M$. The latter has finitely generated higher homotopy groups by \cite[Lemma 1.14]{BKK2024} since its fiber is the Stiefel manifold $V^{d}_{m}$.
\end{proof}
\begin{lemma}\label{lem:cohomology_mli}
  Let $d \geq 3$ and let $\iota: M \hookrightarrow \R^{d-2}$ be an embedding of a closed smooth manifold of dimension $m$. The graded $\FIs$-module $H^*(\oML)$ is of finite type with slope $\leq \frac{d-2}{d-m-2}$.
\end{lemma}
\begin{proof}
  By our grading conventions (see \cref{rmk:conventions}) the underlying graded $\FIs$-module of $C^{*}_{\textnormal{CE}}(\mlgraph)$ is simply $\operatorname{Sym}(\mlgraph)^{\vee}$. Note that here the dual is taken degree-wise. By \cref{lem:computation_graphs_ml} the graded $\FIs$-module $\mlgraph$ is of finite type with slope $\leq \frac{d-2}{d-m-2}$ since
  \[
  \frac{n+d-3}{d-m-2} \leq n\frac{d-2}{d-m-2}, \text{ for } n \geq 1.
  \]
  It is also trivial in degree $0$ by definition. By \cref{prop:algebra_slope} the free symmetric algebra $\operatorname{Sym}(\mlgraph)$ satisfies the statement of the theorem, hence, by \cref{prop:properties_dual} so does $C^{*}_{\textnormal{CE}}(\mlgraph) = \operatorname{Sym}(\mlgraph)^{\vee}$. Finally, since all the structure maps of $\mlgraph$ are strict $C^{*}_{\textnormal{CE}}(\mlgraph)$ is a well-defined dg $\FIs$-module. Therefore, its homology $H^{*}_{\textnormal{CE}}(\mlgraph)$ satisfies the statement of the theorem. The proof is finished by \cref{lem:properties_cube,prop:passing_graphs_ml}.
\end{proof}
\begin{theorem}\label{thm:cohomology_ml}
  Let $d \geq 3$ and let $\iota$ be an embedding of an $m$ dimensional closed manifold $M$ into $\R^{d-2}$. The graded $\FIs$-module $H^*(\ML; \mathbb{Q})$ is of finite type with slope  $\leq \frac{d-2}{d-m-2}$.
\end{theorem}
\begin{proof}
  Our starting point is the homotopy fiber sequence of
  \[
  \oML(T) \to \ML(T) \to \tML(T).
  \]
  Its associated Serre spectral sequence with local coefficients is
  \begin{equation}\label{eq:sss}
    E^{p,q}_{2}(T) = H^{p}(\tML(T); H^{q}(\oML)(T)) \implies H^{p+q}(\ML)(T).
  \end{equation}
  It is natural in $T$. We will show that for fixed $(p,q)$ the resulting $\FIs$-module $E^{p,q}_{2}$ is generated in arity $\leq (p+q)\kappa$, where
  \[
  \kappa \coloneq \frac{d-2}{d-m-2} \geq 1.
  \]
  From this we will conclude by \cref{lem:thick} that $E^{p,q}_{\infty}$ is generated in arity $\leq (p+q)\kappa$ and, hence, the $\FIs$-module
  \[
  H^{n}(\ML) \simeq \bigoplus_{\substack{p+q=n \\ p,q \geq 0}} E^{p,q}_{\infty}
  \]
  is generated in arity $\leq n\kappa$.

  The $\FIs$-module $H^{q}(\oML)$ is generated in arity $\leq q\kappa$. It can be decomposed as
  \begin{equation}\label{eq:homology_decomposition}
    H^{q}(\oML)(T) = \bigoplus_{\substack{S \subset T \\ \lvert S \rvert \leq q\kappa}} \crosseffect H^{q}(\oML)(S).
  \end{equation}
  We wish to show that every summand is a subrepresentation of $\pi_{1}\tML(T)$. Let $\gamma_{T}$ denote an element of $\Omega \tML(T)$. It acts on
  \[
  \oML(T) = \{(x, \eta) \mid x\in\ML(T); \, \eta: [0,1] \to \tML(T), \eta(0)=x, \eta(1)=\iota\}
  \]
  by sending $(x, \eta)$ to $(x, \eta*\gamma_{T})$.

  We start by showing that the structure maps $\phi^{T}_{S}$ and $\psi^{T}_{S}$ behave well with respect to this action. For a subset $S \subset T$ we write $\gamma_{S}$ to denote the image of $\gamma_{T}$ under the projection $\Omega \tML(T) \to \Omega \tML(S)$. The diagram
  \[
  \begin{tikzcd}
    \oML(T) \arrow[r, "\phi^{T}_{S}"] \arrow[d, "\gamma_{T}"] & \oML(S) \arrow[d, "\gamma_{S}"] \\
    \oML(T) \arrow[r, "\phi^{T}_{S}"] & \oML(S)
  \end{tikzcd}
  \]
  is commutative by definition. In the case of $\psi^{T}_{S}$ we wish to show that the diagram
  \[
  \begin{tikzcd}
    \oML(S) \arrow[r, "\psi^{T}_{S}"] \arrow[d, "\gamma_{S}"] & \oML(T) \arrow[d, "\gamma_{T}"] \\
    \oML(S) \arrow[r, "\psi^{T}_{S}"] & \oML(T)
  \end{tikzcd}
  \]
  is commutative up to homotopy, where $\gamma_{S}$ is the projection of $\gamma_{T}$. It is enough to show that in case $S=\underline{r}$, $T=\underline{r+1}$ and arbitrary $\gamma_{\underline{r+1}} = (\gamma_{1},\dots,\gamma_{r+1}) \in \Omega \tML(\underline{r+1})$. We write an element $(x, \eta)$ of $\oML(\underline{r})$ as $((x_{1},\eta_{1}),\dots,(x_{r},\eta_{r}))$ where $(x_{i}, \eta_{i}) \in \oML(\{i\})$. Then
  \begin{align*}
    \gamma_{\underline{r+1}}\psi_{r}(x,\eta) &= ((\Psi_{r}x_{1},\Psi_{r}\eta_{1}*\gamma_{1}),\dots,(\Psi_{r}x_{r},\Psi_{r}\eta_{r}*\gamma_{r}), (\iota, c_{\iota}*\gamma_{r+1})), \\
    \psi_{r}\gamma_{\underline{r}}(x,\eta) &= ((\Psi_{r}x_{1},\Psi_{r}(\eta_{1}*\gamma_{1})),\dots,(\Psi_{r}x_{r},\Psi_{r}(\eta_{r}*\gamma_{r})), (\iota, c_{\iota})).
  \end{align*}
  By definition of $\psi_{r}$ in \cref{cns:fis_sl} there exists an open interval $I$ such that $\iota(M\times\{r+1\})$ is contained in $\R^{d-3} \times I \times \R^{2}$ and for any $x \in \ML(\underline{r})$ the image of $M \times \underline{r}$ under $\psi_{r}(x)$ is disjoint with $\R^{d-3} \times I \times \R^{2}$. On the following commutative diagram
  \[
  \begin{tikzcd}
    \operatorname{hofib} \arrow[r] \arrow[d, "\simeq"] & \operatorname{Emb}_{c}(M\times\{r+1\}, \R^{d-3} \times I \times \R^{2})_{\iota} \arrow[r] \arrow[d, "\simeq"]
    & \tML(\{r+1\}) \arrow[d, equal] \\
    \oML(\{r+1\}) \arrow[r] & \ML(\{r+1\}) \arrow[r] & \tML(\{r+1\})
  \end{tikzcd}
  \]
  the middle vertical map is a weak equivalence given by post-composing with the inclusion $\R^{d-3} \times I \times \R^{2} \subset \R^d$. Hence, the left vertical map is also a weak equivalence. In particular, since $\oML(\{r+1\})$ is path-connected there is a path $h: [0,1] \to \oML(\{r+1\})$ from $(\iota, c_{\iota}*\gamma_{r+1})$ to $(\iota, c_{\iota})$, where $c_{\iota}$ is the constant path at $\iota$, such that for any $t$ the image of the first coordinate $\operatorname{pr}_{1}h(t) \in \ML(\{r+1\})$ of $h(t)$ is contained in $\R^{d-3} \times I \times \R^{2}$. For $i\in\underline{r}$ let $h_{i}$ denote an end point preserving homotopy between $(\Psi_{r}-)*\gamma_{i}$ and $\Psi_{r}(-*\gamma_{i})$, which exists by definition of $\Psi_{r}$. We define a homotopy $\tilde{h}$ between $\gamma_{\underline{r+1}}\psi_{r}$ and $\psi_{r}\gamma_{\underline{r}}$ by setting
  \[
  \tilde{h}(x,\eta,t) = ((\Psi_{r}x_{1},h_{1}(t)),\dots,(\Psi_{r}x_{r},h_{r}(t)), h(t)).
  \]
  It is well-defined since $I$ was chosen in such a way that the element 
  \[
  (\Psi_{r}x_{1},\dots,\Psi_{r}x_{r},\operatorname{pr}_{1}h(t)) \in \ML(\underline{r+1})
  \] 
  is an embedding for any $t$.

  The next step is to show that the decomposition in \eqref{eq:homology_decomposition} is actually a decomposition of representations. Recall that $\phi^{T}_{S}\psi^{T}_{S}\simeq\operatorname{id}_{\oML(S)}$, hence, in cohomology $(\phi^{T}_{S})^{*}: H^{q}(\oML)(S) \to H^{q}(\oML)(T)$ is a direct summand inclusion, and from now on we identify $H^{q}(\oML)(S)$ with this direct summand. Suppressing the indices $T$ and $S$ for readability, for any element $x \in H^{q}(\oML)(S)$ and any $[\gamma_{T}] \in \pi_{1} \tML(T)$ we have
  \[
  \gamma^{*}_{T}\phi^{*}(x) = \phi^{*}\gamma^{*}_{S}(x).
  \]
  Therefore, $H^{q}(\oML)(S)$ is a subrepresentation and $\pi_{1}\tML(T \setminus S)$ acts trivially on it. Similarly to \eqref{eq:homology_decomposition} we can decompose
  \[
  H^{q}(\oML)(S) = \bigoplus_{S' \subset S} \crosseffect H^{q}(\oML)(S').
  \]
  To finish this step it is enough to show that for any proper subset $S' \subset S$ and any $x \in \crosseffect H^{q}(\oML)(S)$ the projection on $\crosseffect H^{q}(\oML)(S')$ of the element $\gamma^{*}_{S}(x)$ is $0$. We will actually show that projection on a bigger subspace $H^{q}(\oML)(S')$ is $0$. Such a projection is given by $(\phi^{S}_{S'})^{*}(\psi^{S}_{S'})^{*}$ and
  \[
  \phi^{*}\psi^{*}\gamma^{*}_{S}(x) = \gamma^{*}_{S}\phi^{*}\psi^{*}(x) = 0.
  \]
  We conclude that $\crosseffect H^{q}(\oML)(S)$ is a subrepresentation of $\pi_{1}\tML(T)$ and that $\pi_{1}\tML(T \setminus S)$ acts trivially on it.

  We decompose the second page \eqref{eq:sss} of the spectral sequence as
  \begin{align*}
    E^{p,q}_{2}(T) &= \bigoplus_{\substack{S \subset T \\ \lvert S \rvert \leq q\kappa}} H^{p}(\tML(T); \crosseffect H^{q}(\oML)(S)) \\
    &= \bigoplus_{\substack{S \subset T \\ \lvert S \rvert \leq q\kappa}} \bigoplus_{i+j=p} H^{i}(\tML(S); \crosseffect H^{q}(\oML)(S)) \tens{} H^{j}(\tML(T \setminus S)),
  \end{align*}
  where for the second equality we used the K\"unneth formula for cohomology with local coefficients. It follows that the $\FIs$-module $E^{p,q}_{2}$ is generated in degree $\leq p+q\kappa \leq (p+q)\kappa$.

  To finish the proof it is enough to show that $H^{n}(\ML(S))$ is finite dimensional vector space for any $n \geq 1$ and $S \subset \Np$. Let $D^{d} \subset \R^{d}$ be a ball of big enough radius such that $\iota(M \times S) \subset D^{d}$. Then
  \[
  \ML(S) = \embi{M \times S} \simeq \operatorname{Emb}_{\emptyset}(M \times S, \overline{D}^{d})_{\iota},
  \]
  where $\operatorname{Emb}_{\emptyset}(M \times S, \overline{D}^{d})_{\iota}$ denote the component of $\iota$ of the space of smooth embeddings of a compact triad-pair $M \times S \subset \overline{D}^{d}$. This space is weakly equivalent to a CW-complex with finitely many cells in each dimension by \cite[Proposition 1.9, Proposition 4.5]{BKK2024}.
\end{proof}

\bibliographystyle{alpha}
\bibliography{refs}

\end{document}